\documentclass[12pt,a4paper,twoside,reqno]{amsart}
\usepackage[a4paper,top=3cm,bottom=3cm,inner=2.5cm,outer=2.5cm]{geometry}

\usepackage[utf8]{inputenc}
\usepackage{amsmath,amssymb,amsthm,amsfonts}
\usepackage[british]{babel}
\usepackage{xcolor}
\usepackage{xparse}
\usepackage{multirow}
\usepackage{enumitem}
\usepackage{pst-node}
\usepackage{xspace}
\usepackage{placeins}
\usepackage{xspace}
\usepackage{fix-cm}
\usepackage[nomessages]{fp}
\usepackage{array}
\usepackage[colorlinks,linkcolor=blue,citecolor=darkgray,urlcolor=darkgray,final,hyperindex,linktoc=page,hyperfootnotes=true]{hyperref}
\newcolumntype{F}{>{$}c<{\hspace{-0.9ex}$}}
\newcolumntype{:}{>{$}m{0.8ex}<{$}}
\newcolumntype{R}{>{$}r<{$}}
\newcolumntype{C}{>{$}c<{$}}
\newcolumntype{L}{>{$}l<{$}}
\newcolumntype{N}{@{}>{$}l<{$}}
\newlength\horspace
\newcommand{\h}[1][1.0]{\hspace{#1\horspace}}
\newlength\verspace
\newlength\negverspace
\usepackage{tikz-cd}
\tikzset{iso/.style={draw=none,every to/.append style={edge node={node [sloped, allow upside down, auto=false]{$\cong$}}}}}
\tikzset{simeq/.style={draw=none,every to/.append style={edge node={node [sloped, allow upside down, auto=false]{$\simeq$}}}}}
\tikzset{simeqS/.style={draw=none,every to/.append style={edge node={node [sloped, allow upside down, auto=false]{$\raisebox{0.8em}{$\simeq$}$}}}}}
\tikzset{simeqSRight/.style={draw=none,every to/.append style={edge node={node [sloped, allow upside down, auto=false]{$\raisebox{-1em}{\rotatebox{180}{$\simeq$}}$}}}}}
\tikzset{aiso/.style={simeqS,preaction={draw,->}}}
\tikzset{aisos/.style={simeqSs,preaction={draw,->}}}
\tikzset{Right/.style={double distance=1.7pt,>={Implies},->}}
\tikzset{twoiso/.style={simeqSRight,preaction={draw,Right}}}
\tikzset{dotdot/.style={dash pattern=on 0.25ex off 0.2ex, dash phase=0ex}}
\tikzset{simeqSs/.style={draw=none,every to/.append style={edge node={node [sloped, allow upside down, auto=false]{$\raisebox{-0.8em}{\rotatebox{180}{$\simeq$}}$}}}}}
\tikzset{RightA/.style={double distance=3.5pt,>={Implies},->},%
	triple/.style={-,preaction={draw,RightA}},%
	quadruple/.style={preaction={draw,RightA,shorten >=0pt},shorten >=1pt,-,double,double distance=0.8pt}}
\tikzset{simS/.style={draw=none,every to/.append style={edge node={node [sloped, allow upside down, auto=false]{$\raisebox{0.8em}{$\sim$}$}}}}}
\tikzset{aeq/.style={simS,preaction={draw,->}}}

\DeclareMathOperator{\bbicolim}{bicolim}

\newcommand{\bicolim}[2]{{\bbicolim}^{#1}\h{#2}}

\newcommand{\sigmabicolim}[1]{\sigma\mbox{-}\h[1.5]\bicolim{}{#1}}

\newtheorem{teor}{Theorem}[section]

\newtheorem{lemma}[teor]{Lemma}
\newtheorem{prop}[teor]{Proposition}
\newtheorem{costr}[teor]{Construction}
\newtheorem{teorintro}{Theorem}

\theoremstyle{definition}
\newtheorem{defne}[teor] {Definition}
\newtheorem{oss}[teor]{Remark}
\newtheorem*{notazione}{Notation}

\newenvironment{cd}{\[\begin{tikzcd}[row sep=7ex,column sep=7ex,ampersand replacement=\&]}{\end{tikzcd}\]\ignorespacesafterend}
\newenvironment{cds}[2]{\[\begin{tikzcd}[row sep=#1ex, column sep=#2ex,ampersand replacement=\&]}{\end{tikzcd}\]\ignorespacesafterend}
\newenvironment{cdN}{\begin{tikzcd}[row sep=7ex,column sep=7ex,ampersand replacement=\&]}{\end{tikzcd}\ignorespacesafterend}
\newenvironment{cdsN}[2]{\begin{tikzcd}[row sep=#1ex, column sep=#2ex,ampersand replacement=\&]}{\end{tikzcd}\ignorespacesafterend}
\newenvironment{eqD}[1]{\begin{equation}\label{#1}}{\end{equation}\ignorespacesafterend}
\newenvironment{eqD*}{\begin{equation*}}{\end{equation*}\ignorespacesafterend}

\def\:{\colon}

\def\phi{\varphi}

\def\dfn#1{{\bfseries\itshape #1}}
\def\predfn#1{{\itshape #1}}
\newcommand{\st}{^{\ast}}

\def\c{\circ}

\DeclareFontFamily{OT1}{pzc}{}
\DeclareFontShape{OT1}{pzc}{m}{it}{<->s*[1.21]pzcmi7t}{}
\DeclareMathAlphabet{\mathpzc}{OT1}{pzc}{m}{it}

\DeclareFontFamily{U}{dutchcal}{\skewchar\font=45}
\DeclareFontShape{U}{dutchcal}{m}{n}{<->s*[1.05] dutchcal-r}{}
\DeclareMathAlphabet{\mathlcal}{U}{dutchcal}{m}{n}

\newcommand{\catfont}[1]{\ensuremath{\mathpzc{#1}}\xspace}
\newcommand{\Cal}[1]{\ensuremath{\mathcal{#1}}\xspace}

\newcommand{\K}{\catfont{K}}

\newcommand{\Cat}{\catfont{Cat}}
\newcommand{\Bicat}{\catfont{Bicat}}
\newcommand{\twocat}{2\catfont{Cat}}

\makeatletter

\newcommand{\Sh}[2][\@nil]{%
	\def\tmp{#1}%
	\ifx\tmp\@nnil{\ensuremath{\catfont{Sh}\hspace{-0.15ex}\left({#2}\right)}}%
	\else{\ensuremath{\catfont{Sh}\hspace{-0.15ex}\left({#2},{#1}\right)}}\fi}

\newcommand{\St}[2][\@nil]{%
	\def\tmp{#1}%
	\ifx\tmp\@nnil{\ensuremath{\catfont{St}\hspace{-0.15ex}\left({#2}\right)}}%
	\else{\ensuremath{\catfont{St}\hspace{-0.15ex}\left({#2},{#1}\right)}}\fi}
\makeatother

\newcommand{\Grothdiag}[1]{\Intdiag{#1}}
\newcommand{\Intdiag}[1]{\ensuremath{\scaleu{\int} \hspace{-0.15ex} #1}}

\newcommand{\x}[1][]{\h[-1]\times_{#1}\h[-1]}
\newcommand{\xp}[2]{\h[-1]{\times}^{#2}_{#1}\h[-1]}
\newcommand{\opn}[1]{\operatorname{#1}}
\newcommand{\id}[1]{\operatorname{id}_{#1}}

\newcommand{\op}{\ensuremath{^{\operatorname{op}}}}

\newcommand{\restr}[2]{{\left.\kern-\nulldelimiterspace {#1}\vphantom{\big|} \right|_{#2}}}

\newcommand{\pr}[1]{\operatorname{pr}_{#1}}

\newcommand{\scaleu}[2][1.2]{{\scalebox{#1}{$#2$}}}

\makeatletter
\newcommand{\ar}[2][]{\xrightarrow[#1]{#2}}
\def\xlongrightarrowfill@{\arrowfill@\relbar\relbar\longrightarrow}
\newcommand{\arr}[2][]{%
	\ext@arrow 0099\xlongrightarrowfill@{#1}{#2}}
\newcommand{\aarr}[2][]{%
	\ext@arrow 0099\xlongrightarrowfill@{#1}{#2}} % in tikz

\newcommand{\aR}[2][]{%
	\ext@arrow 0055{\Rightarrowfill@}{#1}{#2}}
\def\xLongrightarrowfill@{\arrowfill@\Relbar\Relbar\Longrightarrow}
\newcommand{\aRR}[2][]{%
	\ext@arrow 0099\xLongrightarrowfill@{#1}{#2}}
\def\aitofill@{\arrowfill@{\lhook\joinrel\relbar}\relbar\rightarrow}
\newcommand{\aito}[2][]{%
	\ext@arrow 3095\aitofill@{#1}{#2}}
\def\Longaitofill@{\arrowfill@{\lhook\joinrel\relbar\joinrel\relbar}\relbar\rightarrow}
\newcommand{\aitoo}[2][]{%
	\ext@arrow 0099\Longaitofill@{#1}{#2}}

\def\xlongleftarrowfill@{\arrowfill@\longleftarrow\relbar\relbar}
\newcommand{\all}[2][]{%
	\ext@arrow 0099\xlongleftarrowfill@{#1}{#2}}
\newcommand{\aL}[2][]{%
	\ext@arrow 0055{\Leftarrowfill@}{#1}{#2}}
\def\xLongleftarrowfill@{\arrowfill@\Longleftarrow\Relbar\Relbar}
\newcommand{\aLL}[2][]{%
	\ext@arrow 0099\xLongleftarrowfill@{#1}{#2}}

\def\xmapstofill@{\arrowfill@{\mapstochar\relbar}\relbar\rightarrow}
\newcommand{\am}[2][]{%
	\ext@arrow 0395\xmapstofill@{#1}{#2}}
\def\xlongmapstofill@{\arrowfill@\relbar\relbar\longmapsto}
\newcommand{\amm}[2][]{%
	\ext@arrow 0399\xlongmapstofill@{#1}{#2}}

\newcommand{\eqq}{\DOTSB\protect\Relbar\protect\joinrel\Relbar}
\def\xeqqfill@{\arrowfill@\Relbar\Relbar\eqq}
\newcommand{\aeqq}[2][]{%
	\ext@arrow 0099\xeqqfill@{#1}{#2}}

\def\xRrightarrowfill@{\arrowfill@\equiv\equiv\Rrightarrow}
\newcommand{\aM}[2][]{\ext@arrow 0359\xRrightarrowfill@{#1}{#2}}
\newcommand{\Llongrightarrow}{%
	\DOTSB\protect\equiv\protect\joinrel\Rrightarrow}
\def\xLlongrightarrowfill@{\arrowfill@\equiv\equiv\Llongrightarrow}
\newcommand{\aMM}[2][]{%
	\ext@arrow 0099\xLlongrightarrowfill@{#1}{#2}}
\makeatother

\newcommand{\aequi}{\ensuremath{\stackrel{\raisebox{-1ex}{\kern-.3ex$\scriptstyle\sim$}}{\rightarrow}}}
\newcommand{\aequii}{\ensuremath{\stackrel{\raisebox{-1ex}{\kern-.3ex$\scriptstyle\sim$}}{\longrightarrow}}}

\newcommand{\PB}[1]{\arrow[#1,phantom,"\scalebox{1.6}{\color{black}$\lrcorner$}",very near start]}
\newcommand{\Ar}[4][]{\arrow[#2,"{#3}"{#1},""{name=#4, anchor=center}]}
\newcommand{\Ars}[4][]{\arrow[#2,"{#3}"'{#1},""{name=#4, anchor=center}]}
\newcommand{\Arb}[6][]{\arrow[#2,"{#3}"{#1},from=#4,to=#5,shorten <= #6 em, shorten >= #6 em]}
\newcommand{\Arbs}[6][]{\arrow[#2,"{#3}"'{#1},from=#4,to=#5,shorten <= #6 em, shorten >= #6 em]}

\NewDocumentEnvironment{cdl}{s O{7} O{7} b}{%
	\IfBooleanF{#1}{\begin{equation*}}\begin{tikzcd}[row sep=#2ex,column sep=#3ex,ampersand replacement=\&]
			#4
		\end{tikzcd}\IfBooleanF{#1}{\end{equation*}}\ignorespacesafterend}{}
\NewDocumentCommand{\csq}{s O{n} O{7} O{7} O{} O{2.7} O{2.2} O{0.5} O{n}}{%
	\def\foocsq##1##2##3##4##5##6##7##8{%
		\IfBooleanTF{#1}{\begin{cdl}*}{\begin{cdl}}[#3][#4]
				{##1}\ifx#2p{\PB{rd}}\fi\arrow[r,"{##5}"]\ifx#9l{\arrow[d,equal,"{##6}"']}\else{\arrow[d,"{##6}"']}\fi\&{##2}\ifx#9r{\arrow[d,equal,"{##7}"]}\else{\arrow[d,"{##7}"]}\fi\ifx#2l{\arrow[ld,Rightarrow,shorten <=#6ex,shorten >=#7ex,"{#5}"{pos=#8}]}\fi\\
				{##3}\ifx#9d{\arrow[r,equal,"{##8}"']}\else{\arrow[r,"{##8}"']}\fi\ifx#2o{\arrow[ur,Rightarrow,shorten <=#6ex,shorten >=#7ex,"{#5}"{pos=#8}]}\fi\&{##4}
		\end{cdl}}%
		\foocsq
}

\newcommand{\commaunivvN}[9][]{%
	\def\foocommaunivvN##1##2##3##4##5{%
		\begin{cdN}
			#2\arrow[rrd,bend left,"{#3}",""'{name=A}]\arrow[rdd,bend right=35,"{#4}"',""{name=B}]\arrow[rd,"{#5}"{#1}]\&[-4ex]\\[-4ex]
			\&#6 \arrow[r,"{##1}"] \arrow[d,"{##2}"'] \& #7 \arrow[ld,Rightarrow,"{##5}",shorten <=2.9ex,shorten >=2.5ex] \arrow[d,"{##3}"] \\
			\&#8 \arrow[r,"{##4}"'] \& #9
	\end{cdN}}%
	\foocommaunivvN%
}

\newcommand{\bicommaunivvN}[9][]{%
	\def\foobicommaunivvN##1##2##3##4##5##6##7{%
		\begin{cdN}
			#2\arrow[rrd,bend left,"{#3}",""'{name=A}]\arrow[rdd,bend right=35,"{#4}"',""{name=B}]\arrow[rd,"{#5}"{#1}]\&[-4ex]\\[-4ex]
			\&#6 \arrow[from=A,Rightarrow,"{##6}"{pos=0.4},shorten <=0.8ex,shorten >=0.8ex,]\arrow[to=B,Rightarrow,"{##7}",shorten <=0.3ex,shorten >=0.3ex,]\arrow[r,"{##1}"] \arrow[d,"{##2}"'] \& #7 \arrow[ld,Rightarrow,"{##5}",shorten <=2.9ex,shorten >=2.5ex] \arrow[d,"{##3}"] \\
			\&#8 \arrow[r,"{##4}"'] \& #9
	\end{cdN}}%
	\foobicommaunivvN%
}

\newcommand{\biisocommaN}[9][0.5]{%
	\def\foobiisocommaN##1##2##3{%
		\begin{cdsN}{5.5}{5.5}
			#2 \arrow[r,"{#6}"] \arrow[d,"{#7}"'] \& #3 \arrow[ld,twoiso,shorten <=##2ex,shorten >=##3ex,"{##1}"{pos=#1}] \arrow[d,"{#8}"]\\
			#4 \arrow[r,"{#9}"'] \& #5
	\end{cdsN}}%
	\foobiisocommaN%
}

\NewDocumentCommand{\twonats}{s O{2.2} O{8} O{7} O{1.05} O{3.45} O{2}}{%
	\def\footwonats##1##2##3##4##5##6##7##8##9{%
		\def\foofootwonats####1####2####3####4####5{%
			\IfBooleanTF{#1}{\begin{cdl}*}{\begin{cdl}}[#3][#4]
					##1 \Ar{r}{##9}{} \Ars{d,bend right=40}{##5}{A} \Ar{d,bend left=40}{##6}{B} \&
					##2 \Ars{d,bend left}{##8}{Q} \arrow[ld,Rightarrow,shift left=#7,"{####4}"{pos=0.48},shorten <=#5ex, shorten >=#6ex]\&[-2ex]
					##1 \Ar{r}{##9}{} \Ar{d,bend right}{##5}{R} \&
					##2 \Ars{d,bend right=40}{##7}{C} \Ar{d,bend left=40}{##8}{D} \arrow[ld,Rightarrow,shift right=#7,"{####5}"'{pos=0.52},shorten <=#6ex, shorten >=#5ex] \\
					##3 \Ars{r}{####1}{} \&
					##4 \&
					##3 \Ars{r}{####1}{} \&
					##4
					\Arbs{Rightarrow}{\,{####2}}{B}{A}{0.3}
					\Arbs{Rightarrow}{\,{####3}}{D}{C}{0.3}
					\Arb{equal}{}{Q}{R}{#2}
			\end{cdl}}%
			\foofootwonats}\footwonats}

\def\dfn#1{{\bfseries #1}}
\def\predfn#1{{\itshape #1}}

\newcommand{\cX}{\Cal{X}}
\newcommand{\cM}{\Cal{M}}
\newcommand{\cY}{\Cal{Y}}
\newcommand{\cZ}{\Cal{Z}}

\newcommand{\cG}{\Cal{G}}
\newcommand{\cP}{\Cal{P}}
\newcommand{\cR}{\Cal{R}}

\newcommand{\cQ}{\Cal{Q}}
\newcommand{\cU}{\Cal{U}}
\newcommand{\cT}{\Cal{T}}
\newcommand{\cW}{\Cal{W}}

\newcommand{\wt}{\widetilde}

\newcommand{\laxsliceslant}[2]{{\raisebox{.1em}{$#1$}\mkern-1mu\left/_{\mkern-4.1mu\operatorname{lax}}\right.\raisebox{-.25em}{$#2$}}}
\newcommand{\laxslice}[2]{\laxsliceslant{#1}{#2}}
\newcommand{\pssliceslant}[2]{{\raisebox{.1em}{$#1$}\mkern-1mu\left/_{\mkern-4.1mu\operatorname{ps}}\right.\raisebox{-.25em}{$#2$}}}
\newcommand{\psslice}[2]{\pssliceslant{#1}{#2}}

\newcommand{\twoBun}[2]{2\catfont{Bun}_{#1}(#2)}
\newcommand{\qst}[2]{[{#1}/{#2}]}

\newcommand{\clst}[1]{\Cal{B}#1}

\newcommand{\Groth}[1]{\int{\hspace{-0.5ex} #1}}

\begin{document}

\title{Principal 2-bundles and quotient 2-stacks}
\author{Elena Caviglia}
\address{School of Computing and Mathematical Sciences, University of Leicester, United Kingdom}
\email{ec363@leicester.ac.uk}

\begin{abstract}
We generalize principal bundles and quotient stacks to the two-categorical context of bisites. We introduce a notion of principal 2-bundle that makes sense for a 2-category with finite flexible limits, endowed with a bitopology. We then use principal 2-bundles to explicitly construct quotient-pre-2-stacks, which are the analogues of quotient stacks one dimension higher. In order to perform this construction, we prove that principal 2-bundles are closed under iso-comma objects and we restrict ourselves to $(2,1)$-categories. Finally, we prove that, if the bisite is subcanonical and the underlying $(2,1)$-category satisfies some mild conditions, quotient pre-2-stacks are 2-stacks.
\end{abstract}

\maketitle

\setcounter{tocdepth}{1}
\tableofcontents
				
\section*{Introduction}
In this paper, we generalize principal bundles and quotient stacks to the two-categorical context of bisites. This is the generalization to one dimension higher of the theory of generalized principal bundles and quotient stacks that we developed in \cite{Genprinbundquost} in the categorical context of sites. Both these theories will also appear in our PhD thesis \cite{mythesis}.

Principal bundles and quotient stacks are key concepts in geometry and topology. Their importance is due, above all, to their fruitful connections with cohomology groups. Given a nice topological space $X$ and a group $G$, the first cohomology group $H^1(X,G)$ of $X$ with coefficients in $G$ corresponds to the set of isomorphism classes of principal $G$-bundles over $X$. And, as also pointed out by Lurie in \cite{Lurie}, the approach of principal $G$-bundles turns out to be very useful and more effective than the standard one involving cocycles, especially when considering an arbitrary topological space. The collections of principal $G$-bundles over different topological spaces can then be captured and studied together thanks to the classifying stack $\clst{G}$, which is the archetypal quotient stack. 

One of the main motivations that brought us to study categorical generalizations of principal bundles and quotient stacks is the possibility to use such new generalized concepts as tools for the study of cohomology groups in new and broader contexts. We believe that this will have many interesting applications both in mathematics and in mathematical physics.

Principal bundles over topological spaces and over smooth manifolds have been introduced and studied in geometry and topology (see \cite{Milnor}, \cite{Husemoller} and \cite{Steenrod}).  In  \cite{Genprinbundquost}, we produced a categorical generalization of  principal bundles that makes sense in any site, i.e.\ in any category endowed with a Grothendieck topology. The topological group involved in the classical notion of principal bundle is generalized to a group object in the category. And the notion of locally trivial morphism is then internalized by considering pullbacks along the morphisms of a covering family for the Grothendieck topology. This generalized notion of principal bundle recovers the classical notions of principal bundles over topological spaces and smooth manifolds as particular cases.

In recent years, the need of a 2-dimensional notion of principal bundle has appeared in mathematical physics. The particular case of principal bundle over manifolds has been generalized to dimension 2 by Bartels, in their PhD thesis \cite{Bartelsthesis}. They introduced a notion of principal $G$-2-bundle over a smooth manifold, where $G$ is a Lie 2-group and the smooth manifold is seen as a smooth 2-manifold with only identities. The main motivation arose from applications in higher gauge theory; instances of principal 2-bundles of this form can be found in topological quantum field theory and string theory. 

In this paper, we introduce a notion of principal 2-bundle that makes sense in any nice enough 2-category equipped with a two-dimensional Grothendieck topology. More precisely, we use the notion of bitopology introduced by Street in \cite{Streetcharbicatstacks} for bicategories. The internal groups used in the context of sites are here replaced by (coherent) 2-groups and pullbacks are replaced by iso-comma objects. Our notion of principal 2-bundle recovers Bartels' one when considering the 2-category of differentiable categories, i.e.\ categories internal to the category of smooth manifolds, with the covering sieves generated by surjective local diffeomorphisms.  We believe that our general theory of principal 2-bundles could have many applications both in higher gauge theory and in other mathematical contexts. It is also worth noticing that an analogous theory can be developed for bicategories using the same ideas. 

Principal bundles over topological spaces, schemes or differentiable manifolds have also been used to perform an explicit construction of quotient stacks. Instances of this construction can be found in \cite{Neumann}, for the algebraic case, and in \cite{Heinloth}, for the differentiable case. And in \cite{Genprinbundquost}, we performed an analogous construction in the categorical context of sites.  In this paper, we generalize this construction to one dimension higher, producing particular trihomomorphisms from a $(2,1)$-category endowed with a bitopology into the tricategory $\twocat$ of 2-categories. These trihomomorphisms are the analogues of quotient prestacks one dimension higher and will be thus called \predfn{quotient pre-2-stacks}.

Since the classical quotient prestacks are always stacks and the generalized quotient stacks of \cite{Genprinbundquost} are stacks if the site satisfies some mild conditions, it is natural to wonder whether generalized quotient pre-2-stacks satisfy stacky gluing conditions.  
As there was no notion of higher dimensional stack suitable for quotient pre-2-stacks, in our paper \cite{2stacks}, we introduced a notion of 2-stack suitable for a trihomomorphism from a bisite into $\Bicat$. Our definition naturally generalizes the definition of stack given by Street in \cite{Streetcharbicatstacks} for pseudofunctors from a 2-category into $\Cat$. But we also proved a useful characterization (Theorem 3.18 of \cite{2stacks}) in terms of explicit gluing conditions that can be checked more easily in practice. The obtained gluing conditions generalize one dimension higher the usual gluing conditions satisfied by a stack.

Analogously to what we showed in the one-dimensional case in \cite{Genprinbundquost}, we prove that, if the bitopology is subcanonical and the underlying $(2,1)$-category is nice enough, quotient pre-2-stacks are 2-stacks. We thus obtain a notion of quotient 2-stack. This is the main theorem of the paper.

\begin{teorintro}\label{aretwostacksintro}
	Let $\K$ be a bicocomplete $(2,1)$-category with finite flexible limits and such that iso-comma objects preserve bicolimits. Let then $\tau$ be a subcanonical bitopology on $\K$ and $\cX,\cG\in \K$  with $\cG$ an internal 2-group. Then the quotient pre-2-stack $\qst{\cX}{\cG}$ is a 2-stack.
\end{teorintro}

To prove this important result, we use the explicit characterization of 2-stack of our paper \cite{2stacks}, as well as the calculus of colimits in two-dimensional slices developed by Mesiti in \cite{Mesiticolimits} and \cite{Mesitithesis}. 
 
A key result needed for the proof of Theorem \ref{aretwostacksintro} is that every object of a subcanonical bisite can be expressed as some kind of two-dimensional colimit of each covering bisieve over it. We proved this result in \cite{2stacks}.

Quotient stacks are a fundamental concept in algebraic geometry because many stacks that naturally arise in geometric contexts either are quotient stacks or can be constructed using quotient stacks as building blocks. For this reason, we believe that both the generalized quotient stacks of \cite{Genprinbundquost} and the quotient 2-stacks studied in this paper could have many useful and interesting applications. 

Furthermore, our categorical generalizations of principal bundles and quotient stacks could give a substantial contribution towards the development of a cohomology theory of schemes, and more in general of stacks, with coefficients in stacks of abelian 2-groups. This theory would produce new and refined 2-categorical invariants associated to schemes and algebraic stacks, that could solve numerous open problems in algebraic geometry.

\subsection*{Outline of the paper}
In section \ref{sec2bundles}, we introduce and study principal 2-bundles. We start by recalling the notions of equivariant morphisms and 2-cells with respect to actions of internal (coherent) 2-groups. We then construct an action on the iso-comma object of two equivariant morphisms. Moreover, we introduce the notion of \predfn{2-locally trivial morphism} in a bisite. We use this notion to define \predfn{principal 2-bundles} and cells between them (Definition \ref{2bundle}). We conclude the section by proving the key result that principal 2-bundles are stable under iso-comma objects (Proposition \ref{comma2bun}).

In section \ref{secquo2stacks}, we use principal 2-bundles to explicitly construct quotient pre-2-stacks, restricting ourselves to the case of a $(2,1)$-category. We then prove that quotient pre-2-stacks are trihomomorphism into $\twocat$. Finally, we prove the main result of the paper, which is that, if the bisite is subcanonical and the underlying $(2,1)$-category satisfies mild assumptions, our quotient pre-2-stacks are 2-stacks (Theorem \ref{aretwostacks}).
 
\section{Principal 2-bundles}\label{sec2bundles}

% 2-groups, actions, equivariant morphisms and 2-cells
% locally trivial morphisms
% principal 2-bundles
% stability under bicomma objects
In this section we introduce a notion of {principal 2-bundle} in the context of bisites. This notion generalizes the notion of 2-space studied by Bartels in their PhD thesis \cite{Bartelsthesis}.

Throughout this section $\K$ will be a 2-category with all finite flexible limits and $\tau$ will be a Grothendieck bitopology on it. The notion of bitopology used here is the one introduced by Street in \cite{Streetcharbicatstacks} (and recalled also in our paper \cite{2stacks}).

\begin{notazione}
	Throughout this chapter the objects of $\K$ will be denoted by calligraphic letters instead of the usual capital block letters. This is coherent with the choice of other authors using the same concepts. And it is also helpful when comparing the notions presented here to the one-dimensional ones of \cite{Genprinbundquost}.
\end{notazione}

\begin{oss}
	In Section \ref{secquo2stacks}, we will restrict to the case in which $\K$ is a \linebreak[4] $(2,1)$-category as the invertibility of 2-cells is required to construct the higher dimensional analogues of quotient stacks.
\end{oss}

\begin{oss}[iso-comma objects]
	Since $\K$ has all flexible limits (in the sense of \cite{BirdKellyPowerStreet}), it has in particular all iso-comma objects. If the iso-comma objects exists, it is a special representative of the bi-iso comma objects. For this reason, we will formulate some definitions that are bicategorical in nature and naturally involve bi-iso comma objects using iso-comma objects instead. This will offer substantial simplifications in the calculations. The same applies to biproducts and products.
\end{oss}
We aim at introducing a notion of principal 2-bundle that makes sense in $\K$. In this two-dimensional setting instead of an internal group we will consider an internal 2-group. This concept has been introduced by Baez and Lauda in \cite{BaezLauda}, but the idea had already been presented by Joyal and Street in \cite{StreetJoyal}. 

An internal 2-group in the 2-category $\K$ is an object of $\K$ endowed with a multiplication, a neutral element and an inverse morphism that satisfy associativity and unitality only up to fixed coherent 2-isomorphisms. We refer the reader to \cite{BaezLauda} for a complete definition.

We know recall the notion of action of an internal 2-group $\cG\in \K$ on an object $\cX\in \K$. This definition can be found in \cite{Bartelsthesis}.

\begin{defne}\label{action2group}
Let $\cG$ be an internal $2$-group in $\K$ and let $\cX$ be an object of $\K$. An \dfn{action} of $\cG$ on $\cX$ is given by a morphism
$$x\: \cG \x \cX \to \cX$$
together with invertible 2-cells
\begin{eqD*}
\csq*[l][7][7][\mu_x]{\cG \x \cG \x \cX}{\cG \x \cX} {\cG \x \cX} {\cX} {\id{\cG}\x x}{m\x \id{\cX}}{x}{x}
\qquad \text{and} \qquad
\begin{cdN}
	{\cX} \arrow[r,"{e\x \id{\cX}}"] \arrow[rd,"{}",equal,""{name=A}]\& {\cG \x \cX} \arrow[to=A,"{\nu_x}",Rightarrow ,shorten <=0.5ex, shorten >= 0.5ex] \arrow[d,"{x}"] \\
	{} \& {\cX}
\end{cdN}
\end{eqD*}
such that the following equalities of 2-cells hold:
\begin{samepage}
\begin{eqD*}
	\begin{cdN}
		{\cG \x \cG \x \cG \x \cX} \arrow[rd,"{\id{\cG}\x m\x \id{\cX}}"] \arrow[d,"{\id{\cG}\x m\x \id{\cX}}"'] \arrow[rr,"{\id{\cG}\x \id{\cG}\x x}"] \&[-4ex] {} \&[-4ex] {\cG \x \cG \x \cX} \arrow[rd,"{\id{\cG}\x x}"] \arrow[ld,"{\id\x \mu_{x}}", Rightarrow ,shorten <=2.5ex, shorten >= 2.5ex]\& {} \\
		{\cG \x \cG \x \cX} \arrow[rd,"{m\c \id{\cX}}"'] \& {\cG \x \cG \x \cX} \arrow[l,"{\alpha_m\x \id{\cX}}",Rightarrow ,shorten <=0.5ex, shorten >= 0.5ex] \arrow[d,"{m\x \id{\cX}}"] \arrow[rr,"{\id{\cG}\x x}"] \& {} \& {\cG \x \cX} \arrow[d,"{x}"] \arrow[lld,"{\mu_x}",Rightarrow ,shorten <=8.5ex, shorten >= 8.5ex] \\
		{} \& {\cG \x \cX} \arrow[rr,"{x}"'] \& {} \& {\cG}
	\end{cdN}
\end{eqD*}
\begin{eqD}{action1}
	\vspace*{-1.8ex}\hspace*{5.7cm}
	\begin{cdsN}{3}{3}
		\hphantom{.}\arrow[d,equal]\\
		\hphantom{.}
	\end{cdsN}\hspace{5.4cm}
\end{eqD}
\begin{eqD*}
	\begin{cdN}
		{\cG \x \cG \x \cG \x \cX} \arrow[d,"{\id{\cG}\x m\x \id{\cX}}"'] \arrow[rr,"{\id{\cG}\x \id{\cG}\x x}"] \& {} \&[-4ex]{\cG \x \cG \x \cX} \arrow[rd,"{\id{\cG}\x x}"] \arrow[d,"{m\x \id{\cX}}"]\& {} \\
		{\cG \x \cG \x \cX} \arrow[rr,"{\id{\cG}\x x}"]\arrow[rd,"{m\c \id{\cX}}"'] \& {} \& {\cG \x \cX} \arrow[ld,"{\mu_x}",Rightarrow ,shorten <=2.5ex, shorten >= 2.5ex] \arrow[rd,"{x}"] \& {\cG \x \cX} \arrow[l,"{\mu_x}",Rightarrow ,shorten <=2.5ex, shorten >= 2.5ex]  \arrow[d,"{x}"] \\
		{} \& {\cG \x \cX} \arrow[rr,"{x}"'] \& {} \& {\cG}
	\end{cdN}
\end{eqD*}
\end{samepage}
\vspace{5.5mm}
and
\vspace{5.5mm}
\begin{samepage}
	\begin{eqD*}
		\begin{cdN}
			{\cG\x \cX} \arrow[r,"{\id{\cG}\x \id{\cG} \x x}"] \arrow[rd,"{}",equal,""{name=A}] \& {\cG\x \cG \x \cX} \arrow[to=A,"{\alpha_e\x \id{\cX}}"{inner sep=0.001ex},Rightarrow ,shorten <=1ex, shorten >= 1ex,shift left=-0.51ex]\arrow[d,"{m\x \id{\cX}}"] \arrow[r,"{\id{\cG}\x x}"] \& {\cG \x \cX} \arrow[ld,"{\mu_x}",Rightarrow ,shorten <=4.5ex, shorten >= 4.5ex] \arrow[d,"{x}"]\\
			{} \& {\cG \x \cX} \arrow[r,"{x}"']\& {\cX} 
		\end{cdN}
	\end{eqD*}
	\begin{eqD}{action2}
		\vspace*{-1.8ex}\hspace*{5.7cm}
		\begin{cdsN}{3}{3}
			\hphantom{.}\arrow[d,equal]\\
			\hphantom{.}
		\end{cdsN}\hspace{5.4cm}
	\end{eqD}
	\begin{eqD*}
		\begin{cdN}
			{\cG\x \cX} \arrow[r,"{\id{\cG}\x \id{\cG} \x x}"] \arrow[rd,"{}",equal] \& {\cG\x \cG \x \cX} \arrow[d,"{\id\x \nu_x}",Rightarrow ,shorten <=1.5ex, shorten >= 1.5ex] \arrow[rd,"{}",equal]  \arrow[r,"{\id{\cG}\x x}"] \& {\cG \x \cX} \arrow[d,"{x}"]\\
			{} \& {\cG \x \cX} \arrow[r,"{x}"']\& {\cX} 
		\end{cdN}
	\end{eqD*}
\end{samepage}
where $\alpha_m$ and $\alpha_e$ are coherence 2-cells associated to the multiplication and the neutral element of $\cG$ (see \cite{BaezLauda}).
\end{defne}

\begin{oss}
	For the action of an internal 2-group the usual axioms of action are true only up to coherent isomorphisms. These isomorphisms are $\mu_x$ and $\nu_x$ and their coherences are expressed by the equalities (\ref{action1}) and (\ref{action2}).
\end{oss}

We now introduce the notions of \predfn{$\cG$-equivariant morphism} and 
\predfn{$\cG$-equivariant 2-cell} in $\K$. These notions are well-known, but we could not find a reference for them in the literature.

\begin{defne}[$\cG$-equivariant morphism] 
	Let $\cG$ be an internal 2-group in $\K$ that acts on  the objects $\cX$ and $\cY$ of $\K$ with actions $x\: \cG\x \cX \to \cX$ and $y\: \cG\x \cY \to Y$ respectively. A morphism $f\: \cX \to \cY$ is said \dfn{$\cG$-equivariant} if there exists an invertible 2-cell 
	\csq[l][7][7][\lambda_f]{\cG \x \cX}{\cX} {\cG \x \cY} {\cY} {x}{\id{\cG}\x f}{f}{y}
	such that the following equalities hold:
\begin{eqD*}
	\scalebox{0.9}{
		\begin{cdN}
			{} \& {\cG \x \cG \x \cX}  \arrow[d,"{m\x \id{\cX}}"']\arrow[ld,"{\id{\cG} \x \id{\cG} \x f}"',bend right=30]  \arrow[r,"{\id{\cG}\x x}"] \& {\cG \x \cX} \arrow[d,"{x}"]  \arrow[ld,"{\mu_x}",Rightarrow ,shorten <=4.5ex, shorten >= 4.5ex]\\
			{\cG \x\cG \x \cY}  \arrow[rd,"{m\x \id{\cY}}"', bend right=35] \& {\cG \x \cX} \arrow[r,"{x}"'] \arrow[d,"{\id{\cG}\x f}"']\& {\cX} \arrow[d,"{f}"] \arrow[ld,"{\lambda_f}",Rightarrow ,shorten <=4.5ex, shorten >= 4.5ex]\\
			{} \& {\cG \x \cY} \arrow[r,"{y}"']\& {\cY} 
		\end{cdN}
		\h[20] = \h[3]
		\begin{cdN}
			{\cG \x \cG \x \cX} \arrow[d,"{\id{\cG} \x \id{\cG} \x f}"']  \arrow[r,"{\id{\cG}\x x}"] \& {\cG \x \cX} \arrow[d,"{\id{\cY}\x f}"] \arrow[rd,"{x}", bend left=30] \arrow[ld,"{\id{\cG} \x \lambda_f}",Rightarrow ,shorten <=4.5ex, shorten >= 4.5ex]\& {} \\
			{\cG \x \cG \x \cY} \arrow[r,"{\id{\cG}\x y}"'] \arrow[d,"{m\x \id{\cY}}"'] \& {\cG \x \cY} \arrow[ld,"{\mu_y}",Rightarrow ,shorten <=4.5ex, shorten >= 4.5ex] \arrow[d,"{y}"]\& {\cX} \arrow[l,"{\lambda_f}",Rightarrow ,shorten <=1.5ex, shorten >= 1.5ex] \arrow[ld,"{f}",bend left=35] \\
			{\cG \x \cY} \arrow[r,"{y}"'] \& {\cY} \& {} 
	\end{cdN}}
\end{eqD*}

and 

\begin{eqD*}
	\scalebox{0.99}{
		\begin{cdN}
			{\cT \x \cX} \arrow[d,"{\id{\cT}\x f}"'] \arrow[r,"{e\x \id{\cX}}"] \& {\cG \x \cX} \arrow[d,"{\id{\cG}\x f}"']\arrow[r,"{x}"] \& {\cX} \arrow[ld,"{\lambda_f}",Rightarrow ,shorten <=3.5ex, shorten >= 3.5ex] \arrow[d,"{f}"]\\
			{\cT \x \cY} \arrow[rr,"{\pr{2}}"', bend right=40,""{name=A,pos=0.55}] \arrow[r,"{e \x \id{\cY}}"'] \& {\cG \x \cY} \arrow[to=A,"{\nu_y}",Rightarrow ,shorten <=0.5ex, shorten >= 0.5ex]\arrow[r,"{y}"']\& {\cY} 
		\end{cdN}
		\qquad = \qquad
		\begin{cdN}
			{\cT \x \cX} \arrow[rr,"{\pr{2}}"', bend right=40,""{name=B,pos=0.55}] \arrow[d,"{\id{\cT}\x f}"']  \arrow[r,"{e\x \id{\cX}}"] \& {\cG \x \cX} \arrow[to=B,"{\nu_x}",Rightarrow ,shorten <=0.5ex, shorten >= 0.5ex] \arrow[r,"{x}"]\& {\cX}  \arrow[d,"{f}"]\\
			{\cT \x \cY} \arrow[rr,"{\pr{2}}"', bend right=40,""{name=C,pos=0.5}] \& {} \& {\cY} 
	\end{cdN}}
\end{eqD*}
\end{defne}

\begin{oss}
	Notice that every morphism of $\K$ is $\cG$-equivariant when the source and the target are equipped with trivial actions of $\cG$. Having this in mind, we will always think of the objects of $\K$ as equipped with trivial actions of $\cG$ when the action is not specified. 
\end{oss}

\begin{defne}[$\cG$-equivariant 2-cell]
		Let $\cG$ be an internal 2-group in $\K$ that acts on  the objects $\cX$ and $\cY$ of $\K$ with actions $x\: \cG\x \cX \to \cX$ and $y\: \cG\x \cY \to Y$ respectively. Let $f,g\: \cX \to \cY$ be $\cG$-equivariant morphisms. A 2-cell $\gamma\: f \Rightarrow g$ is said \dfn{$\cG$-equivariant} if the following equality holds:
		\begin{eqD*}
			\scalebox{1}{
				\begin{cdN}
					{\cG \x \cX} \arrow[d,"{\id{\cG}\x g}"'] \arrow[r,"{x}"] \& {\cX} \arrow[ld,"{\lambda_g}",Rightarrow ,shorten <=3.5ex, shorten >= 3.5ex] \arrow[d,"{g}"',""{name=B,pos=0.53}] \arrow[d,"{f}", bend left=70 ,""{name=A}]\\
					{\cG \x \cY} \arrow[from=A,to =B,"{\gamma}",Rightarrow ,shorten <=0.5ex, shorten >= 0.001ex] \arrow[r,"{y}"] \& {\cY}
				\end{cdN}
				\h[10] = \h[2]
				\begin{cdN}
					{\cG \x \cX} \arrow[r,"{x}"]  \arrow[d,"{\id{\cG}\x f}",""{name=C}] \arrow[d,"{\id{\cG}\x g}"',""{name=D}, bend right=70] \arrow[from=C,to=D,"{\id{\cG}\x \gamma}"',Rightarrow ,shorten <=0.9ex, shorten >= 0.7ex]\& {\cX} \arrow[d,"{f}"] \arrow[ld,"{\lambda_f}",Rightarrow ,shorten <=3.5ex, shorten >= 3.5ex] \\
					{\cG \x \cY} \arrow[r,"{y}"] \& {\cY}
			\end{cdN}}
		\end{eqD*}
	\end{defne}

The following result might be well-known, but we could not find a reference for it. It generalizes  a well known fact about group objects to one dimension higher.

\begin{prop}\label{pseudomonad}
	Let $\cG$ in $\K$ be an internal 2-group. The pseudofunctor 
	$$\cG \x -\: \K \to \K$$
	 is a \dfn{pseudomonad} with multiplication given by $m\x \id{-}$ and identity given by $e\x \id{-}$. 
	 
	 Moreover, the pseudoalgebras for this pseudomonad are exactly the objects of $\K$ equipped by an action of $\cG$;  the pseudomorphisms between pseudoalgebras are exactly the $\cG$-equivariant morphisms and the pseudotransformations between them are exactly the $\cG$-equivariant 2-cells. 
\end{prop}

\begin{proof}
	The fact that $\cG \x -\: \K \to \K$ is a pseudomonad easily follows from the definition of internal 2-group. It is,then, straightforward to prove that the pseudoalgebras for this pseudomonad are exactly the objects of $\K$ equipped by an action of $\cG$ using the definition of action. Finally, the facts that pseudomorphisms between pseudoalgebras are exactly the $\cG$-equivariant morphisms and that the pseudotransformations between them are exactly the $\cG$-equivariant 2-cells simply follow by the definitions of $\cG$-equivariant morphisms and $\cG$-equivariant 2-cell respectively.
\end{proof}

We now describe how to define an action of $\cG$ on the iso-comma object of two morphisms in $\K$, given actions of $\cG$ on the sources of the morphisms. This action will be largely used throughout the chapter.

\begin{costr} \label{commact}
	Let $\cG$ be an internal 2-group of $\K$ that acts on $\cP\in \K$ with action $p\: \cG\x \cP\to \cP$, on $\cY\in \K$ with action $y\: \cG\x \cY \to \cY$ and on $\cZ\in\K$ with action $z\: \cG\x \cZ \to \cZ$ and let $f:\cP\to \cY$ and $g\: \cZ\to \cY$ be $\cG$-equivariant morphisms. 
	
	Consider, then, the iso-comma object square
	\csq[l][7][7][\alpha^{f,g}]{\cP\x[\cY]\cZ}{\cP}{\cZ}{\cY}{f\st g}{g\st f}{f}{g} 
	We want to define an action of $\cG$ on the iso-comma object $\cP\x[\cY]\cZ$ and we can do this using the morphism $\psi\: \cG\x (\cP\x[\cY]\cZ) \to \cP\x[\cY]\cZ$ induced by the universal property of $\cP\x[\cY]\cZ$ as in the following diagram
	\begin{eqD*}
		\commaunivvN{\cG\x(\cP\x[\cY]\cZ)}{p\c(\id{\cG}\x f\st g)}{z \hspace{0.1ex}\c (\id{\cG}\x g\st f)}{\psi}{\cP\x[\cY]\cZ}{\cP}{\cZ}{\cY}{f\st g}{g\st f}{f}{g}{\alpha^{f,g}}
		\qquad= \quad
		\begin{cdN}
			{\cG\x(\cP\x[\cY]\cZ)} \arrow[d,"{z \hspace{0.1ex}\c (\id{\cG}\x g\st f)}"'] \arrow[r,"{p\c(\id{\cG}\x f\st g)}"] \&[6ex] {\cP} \arrow[d,"{f}"] \arrow[ld,"{\gamma}",twoiso ,shorten <=8.5ex, shorten >= 8.5ex] \\
			{\cZ} \arrow[r,"{g}"'] \&  {\cY} 
		\end{cdN}
	\end{eqD*}
	where the isomorphic 2-cell $\gamma$  is given by the following pasting diagram
		\begin{cd}
		\cG\x(\cP\x[\cY]\cZ) \arrow[r,"\id{\cG}\x f\st g"]  \arrow[d,"\id{\cG}\x g\st f"'] \& \cG\x \cP \arrow[ld,"{\id \x \alpha^{f,g}}",Rightarrow ,shorten <=3.5ex, shorten >= 3.5ex] \arrow[r,"p"] \arrow[d,"\id{\cG}\x f"] \& \cP \arrow[d,"f"] \arrow[ld,"{\lambda_f}",Rightarrow ,shorten <=3.5ex, shorten >= 3.5ex]\\
		\cG\x \cZ \arrow[r,"\id{\cG} \x g"] \arrow[d,"z"'] \& \cG\x \cY  \arrow[r,"y"] \& \cY \arrow[d,equal,""] \arrow[lld,"{{\lambda_g}^{-1}}",Rightarrow ,shorten <=12.5ex, shorten >= 12.5ex]\\
		\cZ \arrow[rr,"g"'] \& \& \cY
	\end{cd}
\end{costr}

\begin{prop}
	The morphism $\psi\: \cG\x (\cP\x[\cY]\cZ) \to \cP\x[\cY]\cZ $ defined in Construction~\ref{commact} is an action of $\cG$ on $\cP\x[\cY]\cZ$.
\end{prop}

\begin{proof}
	We need to prove that there exist coherent invertible 2-cells 
	\begin{eqD*}
		\csq*[l][7][7][\mu_\psi]{\cG \x \cG \x (\cP\x[\cY] \cZ)}{\cG \x (\cP\x[\cY] \cZ)} {\cG \x (\cP\x[\cY] \cZ)} {\cP\x[\cY] \cZ} {\id{\cG}\x \psi}{m\x \id{(\cP\x[\cY] \cZ)}}{\psi}{\psi}
		\qquad \text{and} \qquad
		\begin{cdN}
			{\cP\x[\cY] \cZ} \arrow[r,"{e\x \id{(\cP\x[\cY] \cZ)}}"] \arrow[rd,"{}",equal,""{name=A}]\& {\cG \x (\cP\x[\cY] \cZ)} \arrow[to=A,"{\nu_\psi}",Rightarrow ,shorten <=0.5ex, shorten >= 0.5ex] \arrow[d,"{\psi}"] \\
			{} \& {\cP\x[\cY] \cZ}
		\end{cdN}
	\end{eqD*}
We induce $\mu_\psi$ and $\nu_\psi$ using the two-dimensional universal property of the iso-comma object $\cP\x[\cY] \cZ$. 

To induce $\mu_\psi$ we consider the following isomorphic 2-cells:
\begin{cd}
	{\cG\x \cG\x (\cP\x[\cY]\cZ)} \arrow[rr,"\id{\cG}\x \psi"] \arrow[ddd,"m\x \id{\cP\x[\cY]\cZ}"'] \arrow[rd,"\id{\cG}\x(\id{\cG}\x f\st g)"]\& \& {\cG\x (\cP\x[\cY]\cZ)} \arrow[r,"\psi"] \arrow[d,"\id{\cG}\x f\st g"]\& {\cP\x[\cY]\cZ} \arrow[ddd,"f\st g"] \\
	\& {\cG\x \cG\x \cP} \arrow[r,"\id{\cG}\x p"] \arrow[d,"m\x \id{\cP}"]\& {\cG\x \cP} \arrow[rdd,"p"] \arrow[ld,"{\mu_p}",Rightarrow ,shorten <=4.5ex, shorten >= 4.5ex, shift left=2ex] \& {} \\
	\& {\cG\x \cP}\arrow[rrd,"p"]\& \& \\
	{\cG\x (\cP\x[\cY]\cZ)} \arrow[ru,"\id{\cG}\x f\st g"] \arrow[r,"\psi"'] \&{\cP\x[\cY]\cZ} \arrow[rr,"f\st g"'] \& \& {\cP}  
\end{cd}
and 
\begin{cd}
	{\cG\x \cG\x (\cP\x[\cY]\cZ)} \arrow[rr,"\id{\cG}\x \psi"] \arrow[ddd,"m\x \id{\cP\x[\cY]\cZ}"'] \arrow[rd,"\id{\cG}\x(\id{\cG}\x g \st f)"]\& \& {\cG\x (\cP\x[\cY]\cZ)} \arrow[r,"\psi"] \arrow[d,"\id{\cG}\x g \st f"]\& {\cP\x[\cY]\cZ} \arrow[ddd,"g \st f"] \\
	\& {\cG\x \cG\x \cP} \arrow[r,"\id{\cG}\x z"] \arrow[d,"m\x \id{\cZ}"]\& {\cG\x \cZ} \arrow[ld,"{\mu_z}",Rightarrow ,shorten <=4.5ex, shorten >= 4.5ex, shift left=2ex] \arrow[rdd,"z"] \&  \\
	\& {\cG\x \cZ}\arrow[rrd,"z"]\& \& \\
	{\cG\x (\cP\x[\cY]\cZ)} \arrow[ru,"\id{\cG}\x g \st f"] \arrow[r,"\psi"'] \&{\cP\x[\cY]\cZ} \arrow[rr,"g \st f"'] \& \& {\cZ}  
\end{cd}
Using the definition of $\psi$ and the fact that $f$ and $g$ are $\cG$-equivariant, it is straightforward to prove that these 2-cells are compatible and therefore they induce the desired isomorphic 2-cell $\mu_\psi$.

To induce $\nu_\psi$ we consider the following isomorphic 2-cells
\begin{cd}
	{\cT \x (\cP\x[\cY]\cZ)}  \arrow[rd,"{\id{\cT}\x f\st g}"]\arrow[dd,"{\pr{2}}"']\arrow[rr,"{e\x \id{\cP\x[\cY]\cZ}}"]\&[-3ex] {} \&[-3ex] {\cG \x (\cP\x[\cY]\cZ)} \arrow[d,"{\id{\cG}\x f\st g}"] \arrow[r,"{\psi}"] \& {\cP\x[\cY]\cZ} \arrow[dd,"{f\st g}"] \\
	{} \& {\cT \x \cP} \arrow[rrd,"{\pr{2}}"', bend right=10 ,""{name=A,pos=0.4}]\arrow[r,"{e\x \id{\cP}}"] \& {\cG \x \cP} \arrow[to=A,"{\nu_p}",Rightarrow ,shorten <=0.5ex, shorten >= 0.5ex] \arrow[rd,"{p}"] \& {} \\
	{\cP\x[\cY]\cZ} \arrow[rrr,"{f\st g}"'] \& {} \& {}  \& {\cP} 
\end{cd}
and 
\begin{cd}
	{\cT \x (\cP\x[\cY]\cZ)} \arrow[rd,"{\id{\cT}\x g\st f}"] \arrow[dd,"{\pr{2}}"']\arrow[rr,"{e\x \id{\cP\x[\cY]\cZ}}"]\&[-3ex] {} \&[-3ex] {\cG \x (\cP\x[\cY]\cZ)} \arrow[d,"{\id{\cG}\x g\st f}"] \arrow[r,"{\psi}"] \& {\cP\x[\cY]\cZ} \arrow[dd,"{g\st f}"]\\
	{} \& {\cT \x \cZ}  \arrow[rrd,"{\pr{2}}"', bend right=10 ,""{name=A,pos=0.4}] \arrow[r,"{e\x \id{\cZ}}"] \& {\cG \x \cZ} \arrow[to=A,"{\nu_z}",Rightarrow ,shorten <=0.5ex, shorten >= 0.5ex] \arrow[rd,"{z}"] \& {} \\
	{\cP\x[\cY]\cZ} \arrow[rrr,"{g\st f}"'] \& {} \& {}  \& {\cZ} 
\end{cd}
Using the definition of $\psi$ and the fact that $f$ and $g$ are $\cG$-equivariant, it is straightforward to prove that these 2-cells are compatible and therefore they induce the desired isomorphic 2-cell $\nu_\psi$.

It remains to prove that $\mu_{\psi}$ and $\nu_{\psi}$ satisfy the coherence conditions (\ref{action1}) and (\ref{action2}) of Definition \ref{action2group}. This is straightforward using the fact that the corresponding coherences hold for the actions $p$ and $z$.
\end{proof}

To define \predfn{principal $\cG$-2-bundles}, we will need to use the action induced by the internal multiplication of $\cG$ on a product of the form $\cG \x \cX$.

\begin{prop}
	The morphism 
	$$m\x \id{\cX}\: \cG \x \cG \x \cX \to \cG \x \cX$$
	is an action of $\cG$ on $\cG \x \cX$.
\end{prop} 
\begin{proof}
	To simplify the notation we denote by $\rho$ the morphism $m\x \id{\cX}$. We need to prove that there exist coherent isomorphic 2-cells
	\begin{eqD*}
		\csq*[l][7][7][\mu_\rho]{\cG \x \cG \x \cG \x \cU}{\cG \x \cG \x \cU} {\cG \x \cG \x \cU} {\cG \x \cU} {\id{\cG}\x \rho}{m\x \id{\cG \x \cU}}{\rho}{\rho}
		\qquad \text{and} \qquad
		\begin{cdN}
			{\cG \x \cU} \arrow[r,"{e\x \id{\cG \x \cU}}"] \arrow[rd,"{}",equal,""{name=A}]\& {\cG \x \cG \x \cU} \arrow[to=A,"{\nu_\rho}",Rightarrow ,shorten <=0.5ex, shorten >= 0.5ex] \arrow[d,"{\rho}"] \\
			{} \& {\cG \x \cU}
		\end{cdN}
	\end{eqD*}
	We induce $\mu_\rho$ and $\nu_\rho$ using the two-dimensional universal property of the product $\cG\x\cX$. To induce $\mu_\rho$ we use the following isomorphic compatible 2-cells
	\begin{cd}
		{\cG \x \cG \x \cG \x \cX}  \arrow[dd,"{m\x \id{\cX}}"'] \arrow[r,"{\id{\cG}\x \rho}"] \& {\cG \x \cG \x \cX} \arrow[ldd,"{\alpha_m \x \id{}}",Rightarrow ,shorten <=8.5ex, shorten >= 8.5ex,shift left=-3ex] \arrow[d,"{m\x \id{\cX}}"] \arrow[r,"{\rho}"] \& { \cG \x \cX} \arrow[dd,"{\pr{1}}"]\\
		{} \&  {\cG \c \cX} \arrow[rd,"{\pr{1}}"] \& {}\\
		{\cG \x \cG \x \cX} \arrow[ru,"{m\x \id{\cX}}"]  \arrow[r,"{\rho}"']  \& {\cG \x \cX} \arrow[r,"{\pr{1}}"'] \& {\cG}
	\end{cd}
	and 
	\begin{cd}
		{\cG \x \cG \x \cG \x \cX} \arrow[dd,"{m\x \id{\cX}}"'] \arrow[r,"{\id{\cG}\x \rho}"]  \& {\cG \x \cG \x \cX} \arrow[d,"{}",aiso] \arrow[r,"{\rho}"] \& { \cG \x \cX} \arrow[dd,"{\pr{2}}"]\\
		{} \& {(\cG \x \cG) \x \cX} \arrow[rd,"{\pr{2}}"] \& {}\\
		{\cG \x \cG \x \cX} \arrow[ru,"{}",aiso] \arrow[r,"{\rho}"'] \& {\cG \x \cX} \arrow[r,"{\pr{2}}"'] \& {\cG}
	\end{cd}
To induce $\nu_\rho$ we use the following isomorphic compatible 2-cells
\begin{cd}
	{\cG \x \cX} \arrow[rd,"{}",equal,bend right=30,""{name=A}]\arrow[r,"{e \x \id{\cG}\x \id{\cX}}"',bend right=20]  \arrow[dd,equal] \arrow[r,"{e \x \id{\cG\x \cX}}"] \&[2ex] {\cG \x \cG \x \cX} \arrow[to=A,"{\lambda_m\x \id{\cX}}"{inner sep=0.01ex},Rightarrow ,shorten <=2.5ex, shorten >= 2.5ex,shift left=1.6ex]\arrow[d,"{m\x \id{\cX}}"] \arrow[r,"{\rho}"] \& { \cG \x \cX}  \arrow[dd,"{\pr{1}}"]\\
	{} \&  {\cG \x \cX} \arrow[rd,"{\pr{1}}"]  \& {}\\
	{\cG \x \cX}   \arrow[rr,"{\pr{1}}"'] \&{} \& {\cG}
\end{cd}
and 
\begin{cd}
	{\cG \x \cX}  \arrow[dd,equal] \arrow[r,"{e \x \id{\cG\x \cX}}"] \&[2ex] {\cG \x \cG \x \cX} \arrow[d,"{}",aiso]  \arrow[r,"{\rho}"] \& { \cG \x \cX} \arrow[dd,"{\pr{2}}"]\\
	{} \&  {(\cG \x \cG) \x \cX} \arrow[rd,"{\pr{2}}"] \& {}\\
	{\cG \x \cX}   \arrow[rr,"{\pr{2}}"'] \&{} \& {\cX}
\end{cd}
It remains to prove that $\mu_{\rho}$ and $\nu_{\rho}$ satisfy the coherence conditions (\ref{action1}) and (\ref{action2}) of Definition \ref{action2group}. This can be easily done using the coherence diagrams satisfied by $m$ and $e$ respectively.
\end{proof}

In order to introduce \predfn{principal 2-bundles}, we need to define local triviality of a morphism with respect to the bitopology on $\K$. The idea is to use iso-comma objects in place of the pullbacks used in the one-dimensional context.
\begin{defne}\label{2-loctriv}
	Let $\cX$ and $\cY$ be objects of $\K$ with fixed actions of $\cG$ on them. We say that the morphism $g\: \cY\to \cX$ is \dfn{2-locally trivial} if there exists a bisieve $S\: R\Rightarrow \K(-, \cX)$ in $\tau(\cX)$ such that, for every $f\: \cU \to \cX$ in the bisieve $S$, the iso-comma object 
	\begin{eqD*}
		{\biisocommaN{\cY\x[\cX] \cU}{\cY}{{\cU}}{\cX}{g\st f}{f\st g}{g}{f}{\lambda^{f,g}}{3}{3}}
	\end{eqD*}
	is equivalent to the product $\cG\x \cU$ over $\cU$ via a $\cG$-equivariant equivalence.
\end{defne}

\begin{oss}[iso-comma objects vs bi-iso-comma objects]
	The previous definition could equivalently be given using bi-iso-comma objects instead of iso-comma objects. But since $\K$ has all iso-comma objects and they are representatives of the corresponding bi-iso-comma objects, we can more conveniently ask the required property for iso-comma objects. Notice that we only ask equivalences between the iso-comma objects and the corresponding products. This shows the bicategorical nature of the previous definition.
\end{oss}

\begin{oss}[reformulation of Definition \ref{2-loctriv}]\label{reform}
	We can equivalently reformulate Definition \ref{2-loctriv} by asking the existence of a bisieve $S\: R\Rightarrow \K(-, \cX)$ in $\tau(\cX)$ such that, for every $f\: \cU \to \cX$ in the bisieve $S$, there exists a bi-iso-comma object of the form
		\begin{eqD*}
		\biisocommaN{\cG \x \cU}{\cY}{\cU}{\cX}{h}{\pr{2}}{g}{f}{\delta}{2.5}{2.5}
	\end{eqD*}
This reformulation will be useful for the proof of Proposition \ref{comma2bun}.
\end{oss}
We are now ready to introduce \predfn{principal 2-bundles} and cells between them.

\begin{defne}\label{2bundle}
	Let $\cY$ be an object of $\K$. A \dfn{principal $\cG$-2-bundle over $\cY$} is an object $\cP\in\K$ equipped with an action $p\: \cG\x \cP\to \cP$ and a $\cG$-equivariant  2-locally trivial morphism $\pi_{\cP}\: \cP\to \cY$.

	A \dfn{morphism of principal $\cG$-2-bundles over $\cY$} from $\pi_{\cP}\: \cP\to \cY$ to $\pi_{\cQ}\: \cQ\to \cY$ is a $\cG$-equivariant morphism $\phi\: \cP \to \cQ$ in $\K$ together with an isomorphic $\cG$-equivariant 2-cell  
	\begin{cd}
		{\cP} \arrow[rd,"{\pi_{\cP}}"',""{name=A}] \arrow[rr,"{\phi}"] \&[-3ex] {} \&[-3ex] {\cQ} \arrow[ld,"{\pi_{\cQ}}"]  \arrow[to=A,"{\gamma_{\phi}}",twoiso ,shorten <=2.5ex, shorten >= 2.5ex,shift left=1ex,start anchor={[xshift=-1.5ex]},end anchor={[xshift=-1.5ex]}]\\[3ex]
		{} \& {\cX} \& {}
	\end{cd}
	
	A \dfn{2-cell of principal $\cG$-2-bundles over $\cY$} from $\phi\: \cP \to \cQ$ to $\psi\: \cP \to \cQ$ is a $\cG$-equivariant 2-cell $\Gamma\: \phi \Rightarrow \psi$ such that 
	\begin{eqD*}
	\begin{cdN}
		{\cP} \arrow[rr,"{\phi}", bend left=40,""'{name=G}] \arrow[rd,"{\pi_{\cP}}"',""{name=A}] \arrow[rr,"{\psi}"',""'{name=H}] \arrow[from=G,to=H,"{\Gamma}"{pos=0.4},Rightarrow,shorten <=-0.3ex, shorten >= 0.3ex] \&[-3ex] {} \&[-3ex] {\cQ} \arrow[ld,"{\pi_{\cQ}}"]  \arrow[to=A,"{\gamma_{\psi}}",twoiso ,shorten <=2.5ex, shorten >= 2.5ex,shift left=1ex,start anchor={[xshift=-1.5ex]},end anchor={[xshift=-1.5ex]}]\\[3ex]
		{} \& {\cX} \& {}
	\end{cdN}
	\qquad= \qquad
	\begin{cdN}
		{\cP} \arrow[rd,"{\pi_{\cP}}"',""{name=A}] \arrow[rr,"{\phi}"] \&[-3ex] {} \&[-3ex] {\cQ} \arrow[ld,"{\pi_{\cQ}}"]  \arrow[to=A,"{\gamma_{\phi}}",twoiso ,shorten <=2.5ex, shorten >= 2.5ex,shift left=1ex,start anchor={[xshift=-1.5ex]},end anchor={[xshift=-1.5ex]}]\\[3ex]
		{} \& {\cX} \& {}
	\end{cdN}
\end{eqD*}
Principal $\cG$-2-bundles over $\cY$ together with morphisms and 2-cells between them form a 2-category that will be denoted $\twoBun{\cG}{\cY}$.
\end{defne}

%\begin{oss}
%	Principal $\cG$-2-bundles over $\cX$ together with morphisms and 2-cells of principal $\cG$-2-bundles over $\cX$ form a 2-category that will be denoted $\twoBun{\cG}{\cX}$.
%\end{oss}

We now aim at proving that principal $\cG$-2-bundles are closed under iso-comma objects. This property is the higher dimensional analogue of the closure of principal bundles under pullbacks and it will be very important in order to construct quotient pre-2-stacks in Section \ref{secquo2stacks}. To prove this crucial property, we will need to use the following result.

\begin{prop}\label{liftsbiisocomma}
	The forgetful functor 
	$$U\: \opn{Ps}\mbox{-}(\cG\x -)\mbox{-}\opn{Alg} \to \K$$
	from the 2-category of objects of $\K$ equipped with actions of $\cG$ (that are the pseudoalgebras of the pseudomonad $\cG\x -$ by Proposition \ref{pseudomonad}) to $\K$ lifts bi-iso-comma objects.
\end{prop}

\begin{proof}
	Let $(\cP,p),(\cZ,z)$ and $(\cY,y)$ be in $\opn{Ps}\mbox{-}(\cG\x -)\mbox{-}\opn{Alg}$ and let $f\: \cZ \to \cY$ be $\cG$ equivariant morphisms. Consider the bi-iso-comma object 
	\begin{eqD*}
	\biisocommaN{\cP\x[\cY] \cZ}{\cP}{{\cZ}}{\cY}{g\st f}{f\st g}{g}{f}{\lambda^{f,g}}{3}{3}
	\end{eqD*}
	If $\cP\x[\cY] \cZ$ is endowed with an action of $\cG$ exactly as the one of Construction \ref{commact} (that is still an action in the case of the bi-iso-comma object), then $f\st g$ and $g\st f$ are $\cG$-equivariant by definition of the action and so is the 2-cell $\lambda^{f,g}$. 
 	Let now $(\cQ,q)$ be in $\opn{Ps}\mbox{-}(\cG\x -)\mbox{-}\opn{Alg}$ and consider the morphism $v\: \cQ \to \cP\x[\cY] \cZ$ induced by te universal property of the bi-iso-comma object $\cP\x[\cY] \cZ$ as in the following diagram
 		\begin{eqD*}
 		\bicommaunivvN{\cQ}{r}{s}{v}{\cP\x[\cY]\cZ}{\cP}{\cZ}{\cY}{g\st f}{f\st g}{g}{f}{\lambda^{f,g}}{\Omega_1}{\Omega_2}
 		\qquad = \qquad 
 		\csq*[l][7][7][\Omega]{\cQ}{\cP}{\cZ}{\cY}{r}{s}{g}{f}
 	\end{eqD*}
 If $r,s$ and $\Omega$ are $\cG$-equivariant, then so is the morphism $v$. Indeed, the invertible 2-cell 
 	\begin{eqD*}
 	\csq*[l][7][7][\alpha_v]{\cG \x \cQ}{\cQ}{\cG \x (\cP \x[\cY]\cZ) }{\cP \x[\cY]\cZ}{q}{\id{\cG}\x v}{v}{\psi}
 \end{eqD*}
can be induced using the universal property of the bi-iso-comma $\cP \x[\cY]\cZ$ starting from the following compatible 2-cells

\begin{cd}
	{\cG \x \cQ} \arrow[rd,"{\id{\cG}\x r}",""{name=B}] \arrow[r,"{q}"] \arrow[dd,"{\id{\cG}\x v}"'] \& {\cQ} \arrow[d,"{\alpha_r}",Rightarrow ,shorten <=1.5ex, shorten >= 1.5ex] \arrow[rdd,"{r}",""{name=A}] \arrow[r,"{v}"]\& {\cP \x[\cY]\cZ} \arrow[dd,"{g\st f}"] \arrow[to=A,"{{\Omega}^{-1}_1}",Rightarrow ,shorten <=3.5ex, shorten >= 3.5ex]\\
	{} \& {\cG \x \cP} \arrow[d,"{}",twoiso,shorten <=1.5ex, shorten >= 1.5ex] \arrow[rd,"{p}"] \& {} \\
	{\cG \x (\cP \x[\cY]\cZ)} \arrow[from=B,"{\id{\cG}\x \Omega_1}",Rightarrow ,shorten <=6.5ex, shorten >= 6.5ex] \arrow[ru,"{\id{\cG}\x g\st f}"']\arrow[r,"{\psi}"']\& {\cP \x[\cY]\cZ} \arrow[r,"{g\st f}"'] \& {\cP} 
\end{cd}
and 
\begin{cd}
	{\cG \x \cQ} \arrow[rd,"{\id{\cG}\x s}",""{name=B}] \arrow[r,"{q}"] \arrow[dd,"{\id{\cG}\x v}"'] \& {\cQ} \arrow[d,"{\alpha_s}",Rightarrow ,shorten <=1.5ex, shorten >= 1.5ex]  \arrow[rdd,"{s}",""{name=A}] \arrow[r,"{v}"]\& {\cP \x[\cY]\cZ} \arrow[dd,"{f\st g}"] \arrow[to=A,"{{\Omega}_2}",Rightarrow ,shorten <=3.5ex, shorten >= 3.5ex]\\
	{} \& {\cG \x \cZ} \arrow[d,"{}",twoiso,shorten <=1.5ex, shorten >= 1.5ex]  \arrow[rd,"{z}"] \& {} \\
	{\cG \x (\cP \x[\cY]\cZ)} \arrow[from=B,"{\id{\cG}\x \Omega^{-1}_2}",Rightarrow ,shorten <=6.5ex, shorten >= 6.5ex]\arrow[ru,"{\id{\cG}\x f\st g}"']\arrow[r,"{\psi}"']\& {\cP \x[\cY]\cZ} \arrow[r,"{f\st g}"'] \& {\cZ} 
\end{cd}
It is then easy to prove that $\alpha_v$ satisfies the required coherence conditions.
Analogously, one can easily prove that if the starting data are in $\opn{Ps}\mbox{-}(\cG\x -)\mbox{-}\opn{Alg}$, every 2-cell induced by the universal property of the bi-iso-comma object $\cP \x[\cY] \cZ$ is $\cG$-equivariant. This conclude the proof that $U$ lifts bi-iso-comma objects.
\end{proof}

\begin{oss}
	The previous result of lifting requires an explicit proof as it does not easily follows from the pseudomonadicity of $\cG \x -$ (Proposition \ref{pseudomonad}). A result of lifting of PIE limits for pseudomonads is proved in \cite{BlackwellKellyPower}, but only for the forgetful functor from 2-category of strict algebras (with pseudomorphisms between them). This result is not strong enough for us as we cannot restrict to strict algebras. In fact, the strict algebras correspond with the actions that satisfy the usual properties of action on the nose. 
\end{oss}

We will also need to use the following well-known basic property of the bi-iso-comma object, which is a higher dimensional analogue of a well-known property of pullbacks.

\begin{lemma}\label{lemmabicomma}
	Given the following diagram in $\K$ 
	\begin{cd}
		{\cT} \arrow[r,"{r}"]  \arrow[d,"{s}"']\& {\cM} \arrow[ld,"{\beta}",twoiso ,shorten <=2.5ex, shorten >= 2.5ex] \arrow[d,"{q}"]\arrow[r,"{l}"] \& {\cZ} \arrow[ld,"{\lambda}",twoiso, ,shorten <=2.5ex, shorten >= 2.5ex] \arrow[d,"{f}"]\\
		{\cX} \arrow[r,"{h}"'] \& {\cP} \arrow[r,"{g}"'] \& {\cY}
	\end{cd}
where the isomorphic 2-cell $\lambda$ exhibits a bi-iso-comma object, the isomorphic 2-cell $\beta$ exhibits a bi-iso comma object if and only if the pasting diagram exhibits a bi-iso-comma object.
\end{lemma}

\begin{proof}
	Straightforward using the definition of bi-iso-comma object.
\end{proof}

We are  now ready to prove that principal $\cG$-2-bundles are closed under iso-comma-objects.

\begin{prop} \label{comma2bun}	Let $\pi_\cP\: \cP\to \cY$ be a principal $\cG$-2-bundle  over $\cY$. Then the iso-comma object $\cP\x[\cY]\cZ$ given by 
		\begin{eqD*}
		{\biisocommaN{\cP\x[\cY]\cZ}{\cP}{\cZ}{\cY}{\pi_{\cP}\st f}{f\st \pi_{\cP}}{\pi_{\cP}}{f}{\lambda_{\cP}^{f}}{3}{3}}
	\end{eqD*}
	is a principal $\cG$-2-bundle over $\cZ$ with morphism $\pi_{\cP\x[\cY]\cZ}:=f\st \pi_{\cP}$.
\end{prop}

\begin{proof}
	By Proposition \ref{liftsbiisocomma}, the morphism $f\st \pi_{\cP}$ is $\cG$-equivariant. So it remains to prove that it is 2-locally trivial.
	Since $\pi_{\cP}$ is 2-locally trivial, by Remark \ref{reform}, there exists a covering bisieve $S\: R \Rightarrow \K(-,\cY)$ over $\cY$ such that, for every $g\: \cT \to \cY$ in the bisieve $S$, there exists a bi-iso-comma object of the form 
		\begin{eqD*}
		\biisocommaN{\cG \x \cT}{\cP}{\cT}{\cX}{h}{\pr{2}}{\pi_{\cP}}{g}{\delta}{2.5}{2.5}
	\end{eqD*}
We now consider the  covering bisieve $f\st S$.  It is easy to see that every morphism in $f\st S$ is isomorphic a composite 
$$\cW \ar{v} \cT\x[\cY]\cZ \ar{f\st g} \cZ$$
with $g\: \cT \to \cY$ in $S$. And so it suffices to prove that there exists a bi-iso-comma object of the form 
\begin{cd}
	{\cG \x \cW}  \arrow[rr,"{k}"] \arrow[d,"{\pr{2}}"']\& {} \& {\cP \x[\cY] \cZ} \arrow[lld,"{}", twoiso ,shorten <=9.5ex, shorten >= 9.5ex] \arrow[d,"{f\st \pi_{\cP}}"]\\
	{\cW} \arrow[r,"{v}"'] \& {\cT \x[\cY] \cZ} \arrow[r,"{f\st g}"']\& {\cZ}
\end{cd}
We start observing that the identity 2-cell inside a commutative square of the form
\csq{\cG \x \cU}{\cG \x \cT}{\cU}{\cT}{\id{\cG}\x u}{\pr{2}} {\pr{2}} {u}
exhibits a bi-iso-comma object. Indeed, it is straightforward to prove that $\cG \x \cU$ satisfies the universal properties of the bi-iso-comma object inducing the required morphisms and 2-cells via the universal properties of the product.
Then, thanks to Lemma \ref{lemmabicomma}, it is enough to prove that there exists a bi-iso-comma object of the form 
\begin{cd}
	{\cG \x (\cT \x[\cY] \cZ)} \arrow[r,"{l}"] \arrow[d,"{\pr{2}}"']\& {\cP\x[\cY] \cZ} \arrow[ld,"{}",twoiso ,shorten <=3.5ex, shorten >= 3.5ex] \arrow[d,"{f\st \pi_{\cP}}"] \\
	{\cT \x[\cY] \cZ} \arrow[r,"{f\st g}"']\& {\cZ}
\end{cd}
To prove this, we first look at the following pasting of two bi-iso-comma objects
\begin{cd}
	{\cG  \x (\cT \x[\cY] \cZ)} \arrow[r,"{\id{\cG}\x g\st f}"] \arrow[d,"{\pr{2}}"']\&[3ex] {\cG \x \cT}  \arrow[r,"{h}"]\arrow[d,"{\pr{2}}"']\& {\cP} \arrow[ld,"{\delta}",twoiso,Rightarrow ,shorten <=3.5ex, shorten >= 3.5ex] \arrow[d,"{\pi_{\cP}}"] \\
	{\cT \x[\cY] \cZ} \arrow[r,"{g\st f}"'] \& {\cT} \arrow[r,"{g}"'] \& {\cY}
\end{cd}
By Lemma \ref{lemmabicomma}, it
exhibits a bi-iso comma object. Since the following 2-cell is invertible
	\begin{eqD*}
	{\biisocommaN{\cT\x[\cY]\cZ}{\cT}{\cZ}{\cY}{g\st f}{f\st g}{g}{f}{\lambda^{f,g}}{3}{3}}
\end{eqD*}
the following pasting exhibits a bi-iso-comma object as well
\begin{eqD}{bicomma1}
\begin{cdN}
	{\cG  \x (\cT \x[\cY] \cZ)} \arrow[rr,"{l}"] \arrow[d,"{\pr{2}}"']\&[3ex] {} \& {\cP} \arrow[lld,"{\beta}",twoiso,Rightarrow ,shorten <=9.5ex, shorten >= 9.5ex] \arrow[d,"{\pi_{\cP}}"] \\
	{\cT \x[\cY] \cZ} \arrow[r,"{f\st g}"'] \& {\cZ} \arrow[r,"{f}"'] \& {\cY}
\end{cdN}
\end{eqD}
where $\beta$ is given by the pasting of $\hat{\delta}$ and $\lambda^{f,g}$. 

Moreover, by Lemma \ref{lemmabicomma}, the following pasting of bi-iso-comma object squares
\begin{eqD}{bicomma2}
\begin{cdN}
	{\cM} \arrow[r,"{q_1}"]  \arrow[d,"{q_2}"']\& {\cP \x[\cY] \cZ} \arrow[ld,"{}",twoiso ,shorten <=2.5ex, shorten >= 2.5ex] \arrow[d,"{f\st \pi_{\cP}}"']\arrow[r,"{\pi_{\cP}\st f}"] \& {\cP} \arrow[ld,"{\lambda_{\cP}^f}",twoiso, ,shorten <=2.5ex, shorten >= 2.5ex] \arrow[d,"{\pi_{\cP}}"]\\
	{\cT \x[\cY]\cZ} \arrow[r,"{f\st g}"'] \& {\cZ} \arrow[r,"{f}"'] \& {\cY}
\end{cdN}
\end{eqD}
exhibits a bi-iso-comma object. Since both this bi-iso-comma object and that of diagram \ref{bicomma1} are bi-iso-comma objects of  $\pi_{\cP}$ and $f\c f\st g$, there exists an equivalence 
$$\phi\: \cG \x (\cT \x[\cY] \cZ) \aequi{} \cM$$
And it is straightforward to prove that also the following pasting diagram exhibits a bi-iso-comma object
\begin{eqD}{bicomma3}
\begin{cdN}
	{\cG \x (\cT \x[\cY] \cZ)}  \arrow[rd,"{\pr{2}}"', bend right=20 ,""{name=A}] \arrow[r,"{\phi}"',aeq] \& {\cM} \arrow[to=A,"{}",twoiso ,shorten <=2.5ex, shorten >= 2.5ex] \arrow[r,"{q_1}"]  \arrow[d,"{q_2}"']\& {\cP \x[\cY] \cZ} \arrow[ld,"{}",twoiso ,shorten <=2.5ex, shorten >= 2.5ex] \arrow[d,"{f\st \pi_{\cP}}"']\arrow[r,"{\pi_{\cP}\st f}"] \& {\cP} \arrow[ld,"{\lambda_{\cP}^f}",twoiso, ,shorten <=2.5ex, shorten >= 2.5ex] \arrow[d,"{\pi_{\cP}}"]\\
	{} \& {\cT \x[\cY]\cZ} \arrow[r,"{f\st g}"'] \& {\cZ} \arrow[r,"{f}"'] \& {\cY}
\end{cdN}
\end{eqD}
where the isomorphic two-cell inside the left triangle is the one induced by the universal property of the bi-iso-comma object $\cM$ that induces the equivalence $\phi$. 

Finally, applying again Lemma \ref{lemmabicomma} to diagram \ref{bicomma3} (considering together the two 2-cells on the left), we conclude that the following diagram exhibits a bi-iso-comma object
\begin{cd}
	{\cG \x (\cT \x[\cY] \cZ)}  \arrow[rd,"{\pr{2}}"', bend right=20 ,""{name=A}] \arrow[r,"{\phi}"',aeq] \& {\cM} \arrow[to=A,"{}",twoiso ,shorten <=2.5ex, shorten >= 2.5ex] \arrow[r,"{q_1}"]  \arrow[d,"{q_2}"']\& {\cP \x[\cY] \cZ} \arrow[ld,"{}",twoiso ,shorten <=2.5ex, shorten >= 2.5ex] \arrow[d,"{f\st \pi_{\cP}}"]\\
	{} \& {\cT \x[\cY]\cZ} \arrow[r,"{f\st g}"']  \& {\cZ} 
\end{cd}
This concludes the proof that $f\st \pi_{\cP}$ is a principal $\cG$-2-bundle.
\end{proof}

\section{Quotient 2-stacks} \label{secquo2stacks}
In this section we will explicitly construct analogues of quotient stacks one dimension higher. We will then prove that they are well-defined trihomomorphisms. We will conclude the section proving that, under mild assumptions, these objects are 2-stacks. This will be the main result of the paper.

Throughout this section $\K$ will be a $(2,1)$-category with all finite flexible limits and $\tau$ will be a bitopology on it. As we will explain in detail later on, we need to restrict ourselves to the case of a $(2,1)$ category to be able to generalize the quotient prestacks of \cite{Genprinbundquost} to one dimension higher. 

We start giving the explicit construction of \predfn{quotient pre-2-stacks}. It generalizes the definition of quotient prestacks of \cite{Genprinbundquost} to the two-dimensional context.

\begin{defne}\label{qp2}
	Let $\cX$ be an object of the $(2,1)$-category $\K$ and let $\cG$ be an internal 2-group in $\K$ that acts on it. The \dfn{quotient pre-2-stack} 
	$$\qst{\cX}{\cG}\: \K\op \to \twocat$$
	is defined as follows:
	\begin{enumerate}
		\item[$\bullet$] for every object $\cY\in \K$ we define $\qst{\cX}{\cG}(\cY)$ as the 2-category that has:
		\begin{itemize}
			\item[-]as objects the pairs $(\cP,\alpha)$ where $\pi_{\cP}\: \cP\to \cY$ is a principal $\cG$-2-bundle over $\cY$ and $\alpha\: \cP\to \cX$ is a $\cG$-equivariant morphism;
			\item[-] as morphisms from $(\cP,\alpha_{\cP})$  to $(\cQ,\alpha_{\cQ})$ the morphisms of principal $\cG$-2-bundles $\phi\: \cP \to \cQ$ such that there exists an invertible 2-cell
				\begin{cd}
				{\cP} \arrow[rd,"{\alpha_{\cP}}"',""{name=A}] \arrow[rr,"{\phi}"] \&[-3ex] {} \&[-3ex] {\cQ} \arrow[ld,"{\alpha_{\cQ}}"]  \arrow[to=A,"{\beta_{\phi}}",twoiso ,shorten <=2.5ex, shorten >= 2.5ex,shift left=1ex,start anchor={[xshift=-1.5ex]},end anchor={[xshift=-1.5ex]}]\\[3ex]
				{} \& {\cX} \& {}
			\end{cd}
		\item[-] as 2-cells from $\phi\: (\cP , \alpha_{\cP}) \to  (\cQ , \alpha_{\cQ})$ to $\psi\: (\cP , \alpha_{\cP}) \to  (\cQ , \alpha_{\cQ})$ the 2-cells of principal $\cG$-2-bundles $\Gamma\: \phi \Rightarrow \psi$ such that 
		\begin{eqD*}
			\begin{cdN}
				{\cP} \arrow[rr,"{\phi}", bend left=40,""'{name=G}] \arrow[rd,"{\alpha_{\cP}}"',""{name=A}] \arrow[rr,"{\psi}"',""'{name=H}] \arrow[from=G,to=H,"{\Gamma}"{pos=0.4},Rightarrow,shorten <=-0.3ex, shorten >= 0.3ex] \&[-3ex] {} \&[-3ex] {\cQ} \arrow[ld,"{\alpha_{\cQ}}"]  \arrow[to=A,"{\beta_{\psi}}",twoiso ,shorten <=2.5ex, shorten >= 2.5ex,shift left=1ex,start anchor={[xshift=-1.5ex]},end anchor={[xshift=-1.5ex]}]\\[3ex]
				{} \& {\cX} \& {}
			\end{cdN}
			\qquad= \qquad
			\begin{cdN}
				{\cP} \arrow[rd,"{\alpha_{\cP}}"',""{name=A}] \arrow[rr,"{\phi}"] \&[-3ex] {} \&[-3ex] {\cQ} \arrow[ld,"{\alpha_{\cQ}}"]  \arrow[to=A,"{\beta_{\phi}}",twoiso ,shorten <=2.5ex, shorten >= 2.5ex,shift left=1ex,start anchor={[xshift=-1.5ex]},end anchor={[xshift=-1.5ex]}]\\[3ex]
				{} \& {\cX} \& {}
			\end{cdN}
		\end{eqD*}
		\end{itemize}
		\item[$\bullet$] for every morphism $f\: \cZ\to \cY$ in $\K$, we define 
		$$\qst{\cX}{\cG}(f)\: \qst{\cX}{\cG}(\cY) \to \qst{\cX}{\cG}(\cZ)\vspace{-3.5mm}$$
		via iso-comma object along $f$ as the 2-functor that sends:
		\begin{itemize}
			\item[-]an object $(\cP,\alpha_{\cP})$ to the pair $(\cP \x[\cY] \cZ,\alpha_{\cP}\c \pi_{\cP}\st f)$, where $\cP \x[\cY] \cZ$ is given by the iso-comma object
				\begin{eqD*}
				{\biisocommaN{\cP\x[\cY]\cZ}{\cP}{\cZ}{\cY}{\pi_{\cP}\st f}{f\st \pi_{\cP}}{\pi_{\cP}}{f}{\lambda_{\cP}^{f}}{3}{3}}
			\end{eqD*}
		and it is equipped with the morphism $f\st \pi_{\cP}$;
		\item[-] a morphism $\phi\: (\cP , \alpha_{\cP}) \to  (\cQ , \alpha_{\cQ})$ to the morphism $\qst{\cX}{\cG}(f)(\phi)$ induced by the universal property of the iso-comma object as in the following diagram 
		\begin{eqD*}
			\commaunivvN{\cP\x[\cY]\cZ}{\phi\c \pi_{\cP}\st f}{f\st \pi_{\cP}}{\qst{\cX}{\cG}(f)(\phi)}{\cQ\x[\cY]\cZ}{\cQ}{\cZ}{\cY}{\pi_{\cQ}\st f}{f\st \pi_{\cQ}}{\pi_{\cQ}}{f}{\alpha_{\cQ}^{f}}
			\qquad= \quad
			\begin{cdN}
				{\cP\x[\cY]\cZ} \arrow[d,"{f\st \pi_{\cP}}"'] \arrow[r,"{\pi_{\cP}\st f}"] \& {\cP}\arrow[r,"{\phi}"] \arrow[rd,"{\pi_{\cP}}"',""{name=A}]  \arrow[ld,"{\alpha_{\cP}^f}",Rightarrow ,shorten <=3.5ex, shorten >= 3.5ex]\& {\cQ} \arrow[d,"{\pi_{\cQ}}"] \arrow[to=A,"{\gamma_\phi}",Rightarrow ,shorten <=0.5ex, shorten >= 0.5ex]\\
				{\cZ} \arrow[rr,"{f}"']\& {} \& {\cY}
			\end{cdN}
		\end{eqD*}
	\item[-] a 2-cell $\Gamma\: \phi \Rightarrow \psi$  from $\phi\: \cP \to \cQ$ to $\psi\: \cP \to \cQ$ to the 2-cell 
	$$\qst{\cX}{\cG}(f)(\Gamma)\: \qst{\cX}{\cG}(f)(\phi) \Rightarrow\qst{\cX}{\cG}(f)(\psi)$$
	induced by the two-dimensional universal property of the iso comma object $\cP\x[\cY]\cZ$ using the following compatible 2-cells
		\begin{eqD*}
			\hspace{-3ex}
		\begin{cdN}
			{\cP \x[\cY] \cZ} \arrow[rd,"{\pi_{\cP}\st f}"] \arrow[dd,"{\qst{\cX}{\cG}(f)(\psi)}"'] \arrow[rr,"{\qst{\cX}{\cG}(f)(\phi)}"]\&[-2ex] {}\&[-2ex]  {\cQ \x[\cY] \cZ} \arrow[dd,"{\pi_{\cQ}\st f}"]\\[-2ex]
			{} \& {\cP} \arrow[rd,"{\phi}",bend left=30,""'{name=A}] \arrow[rd,"{\psi}"',bend right=30,""{name=B}] \arrow[from=A,to=B,"{\Gamma}",Rightarrow ] \& {} \\[-2ex]
			{\cQ \x[\cY] \cZ} \arrow[rr,"{\pi_{\cQ}\st f}"']\& {} \& {\cQ}
		\end{cdN}
		\qquad
		\begin{cdN}
			{\cP \x[\cY] \cZ} \arrow[rd,"{f\st \pi_{\cP}}"]\arrow[d,"{\qst{\cX}{\cG}(f)(\psi)}"'] \arrow[r,"{\qst{\cX}{\cG}(f)(\phi)}"]\&[7ex] {\cQ \x[\cY] \cZ} \arrow[d,"{f\st \pi_{\cQ}}"] \\[6ex]
			{\cQ \x[\cY] \cZ} \arrow[r,"{f\st \pi_{\cQ}}"'] \& {\cZ}
		\end{cdN}
	\end{eqD*} 
				\end{itemize}
		\item[$\bullet$] for every 2-cell $\Lambda\: f \Rightarrow g\: \cZ \to \cY$ in $\K$, we define 
		$$\qst{\cX}{\cG}(\Lambda)\: \qst{\cX}{\cG}(f) \Rightarrow \qst{\cX}{\cG}(g) $$
		 as the 2-natural transformation that has component $\qst{\cX}{\cG}(\Lambda)_{\cP}$ relative to $(\cP, \alpha_{\cP})\in \qst{\cX}{\cG}(\cY)$ induced by the universal property of the iso-comma object $\cP \xp{\cY}{g}\cZ$ as follows
		 \begin{eqD*}
		 	\commaunivvN{\cP\xp{\cY}{f}\cZ}{\pi_{\cP}\st f}{f\st \pi_{\cP}}{\qst{\cX}{\cG}(\Lambda)_{(\cP,\alpha_{\cP})}}{\cP\xp{\cY}{g}\cZ}{\cP}{\cZ}{\cY}{\pi_{\cP}\st g}{g\st \pi_{\cP}}{\pi_{\cP}}{g}{\alpha_{\cP}^{g}}
		 	\qquad= \quad
		 	\begin{cdN}
		 		{\cP\xp{\cY}{f}\cZ} \arrow[d,"{f\st \pi_{\cP}}"'] \arrow[r,"{\pi_{\cP}\st f}"] \&[5ex] {\cP} \arrow[d,"{\pi_{\cP}}"] \arrow[ld,"{\alpha_{\cP}^{f}}"',shift left=-1.3ex,Rightarrow ,shorten <=5.5ex, shorten >= 5.5ex]\\[3ex]
		 		{\cZ} \arrow[r,"{g}"',""{name=A}] \arrow[r,"{f}", bend left=40 ,""'{name=B}]  \arrow[from=B,to=A,Rightarrow,"{\Lambda}"]\& {\cY}
		 	\end{cdN}
		 \end{eqD*}
		 \end{enumerate}
\end{defne}

\begin{oss}
	Notice that the hypothesis that all the 2-cells of $\K$ are isomorphic is needed to define both the action $\qst{\cX}{\cG}$ on morphisms and 2-cells. Indeed, in order to use the universal property of iso-comma objects to induce the desired morphisms we need to start from isomorphic 2-cells.
\end{oss}

We now aim at proving that the quotient pre-2-stacks of Definition \ref{qp2} are well-defined and that they are trihomomorphisms from $\K\op$ into $\twocat$. We will split the proof in various intermediate results.

\begin{prop}\label{twocatobj}
	Let $\cY\in \K$. Then $\qst{\cX}{\cG}(\cY)$ (defined in \ref{qp2}) is a 2-category.
\end{prop}

\begin{proof}
Straightforward using the fact that $\twoBun{\cG}{\cY}$ is a 2-category and the additional data given by the morphisms into the fixed object $\cX$ behave well with respect to identity and composition.
\end{proof}

\begin{prop}\label{twocatmor}
	Let $f:\cZ \to \cY$ be a morphism in $\K$. Then $\qst{\cX}{\cG}(f)$ (defined in \ref{qp2}) is a 2-functor.
\end{prop}

\begin{proof}
	The fact that $\qst{\cX}{\cG}(f)$ is well-defined follows by the closure of principal-$\cG$-2-bundles under iso-comma-objects (Proposition \ref{comma2bun}). We now prove the \linebreak[4] 2-functoriality of $\qst{\cX}{\cG}(f)$. 
	
	Firstly, we observe that, using the two-dimensional universal property of the iso-comma object, it is straightforward to prove that the assignment on 2-cells preserves identities and composition.
	
	Let now $(\cP,\alpha_{\cP})$ be an object of $\qst{\cX}{\cG}(\cY)$. Since both $\id{(\cP \x[\cY]\cZ,\alpha_{\cP} \x \pi_{\cP}\st f)}$ and \linebreak[4] $\qst{\cX}{\cG}(f)(\id{(\cP,\alpha_{\cP})})$ can be induced by the universal property of $\cP \x[\cY]\cZ$ using the same 2-cell, they are equal. So $\qst{\cX}{\cG}(f)$ preserves identities.
	
	Let now $\rho\: (\cR,\alpha_{\cR})\to (\cP,\alpha_{\cP})$ and $\phi\: (\cP,\alpha_{\cP})\to (\cQ,\alpha_{\cQ})$ be morphisms in $\qst{\cX}{\cG}(\cY)$. Since both $\qst{\cX}{\cG}(f)(\phi) \c \qst{\cX}{\cG}(f)(\rho)$ and $\qst{\cX}{\cG}(f)(\phi\c \rho)$ can be induced by the universal property of $\cQ \x[\cY]\cZ$ using the same 2-cell, they are equal. So $\qst{\cX}{\cG}(f)$ preserves composition.
	\end{proof}

\begin{prop}\label{twocat2cells}
	Let $\Lambda\: f \Rightarrow g$ be a 2-cell in $\K$ from $f\: \cZ \to \cY$ to $g\: \cZ \to \cY.$ Then $\qst{\cX}{\cG}(\Lambda)$ (defined in \ref{qp2}) is a 2-natural transformation.
\end{prop}
\begin{proof}
	Let $\phi\: (\cP, \alpha_{\cP}) \to (\cQ, \alpha_{\cQ})$ be a morphism in $\qst{\cX}{\cG}(\cY)$. We need to prove that the following diagram is commutative:
		\csq[][12][12]{\cP \xp{\cY}{f}\cZ}{\cP \xp{\cY}{g}\cZ}{\cQ \xp{\cY}{f}\cZ}{\cQ \xp{\cY}{g}\cZ}{\qst{\cX}{\cG}(\Lambda)_{(\cP, \alpha_{\cP})}}{\qst{\cX}{\cG}(f)(\phi)} {\qst{\cX}{\cG}(g)(\phi)}{\qst{\cX}{\cG}(\Lambda)_{(\cQ, \alpha_{\cQ})}}
Thanks to the universal property of the iso-comma object $\cQ \xp{\cY}{g}\cZ$, this is implied by the fact that the following diagrams are commutative by construction:
\begin{cd}
	{\cP \xp{\cY}{f}\cZ} \arrow[dd,"{\qst{\cX}{\cG}(f)(\phi)}"'] \arrow[rd,"{\pi_{\cP}\st f}"'] \arrow[r,"{\qst{\cX}{\cG}(\Lambda)_{(\cP, \alpha_{\cP})}}"] \&[5ex] {\cP \xp{\cY}{g}\cZ} \arrow[d,"{\pi_{\cP}\st g}"] \arrow[r,"{\qst{\cX}{\cG}(g)(\phi)}"] \&[5ex] {\cQ \xp{\cY}{g}\cZ} \arrow[dd,"{\pi_{\cQ}\st g}"] \\
	{} \& {\cP} \arrow[rd,"{\phi}"] \& {}\\
	{\cQ \xp{\cY}{f}\cZ} \arrow[rr,"{\pi_{\cQ}\st f}", bend left=20] \arrow[r,"{\qst{\cX}{\cG}(g)(\phi)}"'] \& {\cQ \xp{\cY}{g}\cZ} \arrow[r,"{\pi_{\cQ}\st g}"'] \& {\cQ}
\end{cd}
and 
\begin{cd}
	{\cP \xp{\cY}{f}\cZ}  \arrow[rrdd,"{f\st \pi_{\cP}}"]\arrow[dd,"{\qst{\cX}{\cG}(f)(\phi)}"'] \arrow[r,"{\qst{\cX}{\cG}(\Lambda)_{(\cP, \alpha_{\cP})}}"] \&[5ex] {\cP \xp{\cY}{g}\cZ} \arrow[rdd,"{g\st \pi_{\cP}}"] \arrow[r,"{\qst{\cX}{\cG}(g)(\phi)}"] \&[5ex] {\cQ \xp{\cY}{g}\cZ} \arrow[dd,"{g\st \pi_{\cQ}}"] \\
	{} \& {} \& {}\\
	{\cQ \xp{\cY}{f}\cZ} \arrow[rr,"{f\st \pi_{\cQ}}", bend left=20] \arrow[r,"{\qst{\cX}{\cG}(g)(\phi)}"'] \& {\cQ \xp{\cY}{g}\cZ} \arrow[r,"{g\st \pi_{\cQ}}"']\& {\cZ}
\end{cd}
\end{proof}

\begin{oss}
	The previous three propositions (\ref{twocatobj}, \ref{twocatmor} and \ref{twocat2cells}) show that $\qst{\cX}{\cG}$ takes values in $\twocat$.
\end{oss}

Before proving that $\qst{\cX}{\cG}$ is a trihomomorphism, we recall the definition of trihomomorphism. We take \cite{Gurskithesis} as main reference. 
\begin{defne}\label{trihomomorphism}
	Let $T$ and $T'$ be tricategories. A \dfn{trihomomorphism} $F\: T \to T'$ is given by the following data:
	
	\begin{itemize}
		\item a function $F:\opn{Ob}(T) \to \opn{Ob}(T')$;
		\item given $a,b\in T$ a pseudofunctor 
		$F_{a,b}\: T(a,b) \to T'(F(a),F(b));$
		\item given $a,b,c,d\in T$ an adjoint equivalence $\chi\: \otimes' \c (F\x F) \Rightarrow F \c \otimes $ with left adjoint 
		\begin{cd}
			{T(b,c) \x T(a,b)} \arrow[r,"{F\x F}"] \arrow[d,"{\otimes}"'] \& {T'(F(b),F(c))\x T'(F(a), F(b))}\arrow[d,"{ \otimes'}"] \arrow[ld,"{\chi}", Rightarrow, shorten <= 6ex, shorten >= 7ex ]\\
			{T(a,c)} \arrow[r,"{F}"', ]\& {T'(F(a),F(c));} 
		\end{cd}
		\item given $a\in T$ an adjoint equivalence $\iota\: I'_{F(a)} \Rightarrow F\c I_a$ with left adjoint 
		\begin{cd}
			{1} \arrow[rd,"{I'_{F(a)}}"',""{name=G}] \arrow[r,"{I_a}"]\& {T(a,a)}\arrow[to=G,"{\iota}", Leftarrow] \arrow[d,"{F}"]\\
			{} \& {T'(F(a),F(a));}
		\end{cd}
	\item given $a,b,c,d\in T$  an invertible modification 
	\begin{eqD*}
		\begin{cdsN}{3.4}{2.5}
			{} \& {T^3} \arrow[rr,"{F \x F\x F}"] \arrow[ld,"{1 \x \otimes}"'] \arrow[rd,"{ \otimes \x 1 }"]\& {} \& {T'^3} \arrow[ld,"{\chi\x 1}", Rightarrow, shorten <=1.2ex,shorten >=1.2ex ]\arrow[rd,"{1 \x \otimes'}"] \& {} \\
			{T^2} \arrow[rd,"{1 \x \otimes}"'] \& {} \& {T^2} \arrow[ll,"{ \alpha}",Rightarrow, shorten <=3.2ex,shorten >=3.2ex] \arrow[rr,"{F\x F}"'] \arrow[ld,"{ \otimes}"]\& {} \& {T'^2} \arrow[llld,"{\chi}", Rightarrow, shorten <=6.7ex,shorten >=6.7ex] \arrow[ld,"{\otimes'}"] \\
			{} \& {T} \arrow[rr,"{F}"'] \& {} \& {T'} \& {} 
		\end{cdsN}
		\aM{\omega}
		\begin{cdsN}{3.4}{2.5}
			{} \& {T^3} \arrow[rr,"{F\x F \x F}"] \arrow[ld,"{1\x \otimes}"'] \& {} \& {T'^3} \arrow[llld,"{1 \x \chi}", shorten <=6.5ex, shorten >= 6.5 ex, Rightarrow]\arrow[ld,"{1 \x \otimes'}"]\arrow[rd,"{\otimes' \x 1}"] \& {} \\
			{T^2} \arrow[rr,"{F\x F}"'] \arrow[rd,"{\otimes}"'] \& {} \& {T'^2} \arrow[ld,"{\chi}", Rightarrow, shorten <= 1ex, shorten >= 1ex] \arrow[rd,"{\otimes'}"] \& {} \& {T'^2} \arrow[ll,"{\alpha'}", Rightarrow, shorten <=2.5 ex, shorten >= 2.5 ex] \arrow[ld,"{\otimes'}"] \\
			{} \& {T} \arrow[rr,"{F}"'] \& {} \& {T';} \& {} 
		\end{cdsN}
	\end{eqD*}
\item given $a,b\in T$ invertible modifications
\begin{eqD*}
\begin{cdsN}{5.5}{5.5}
	{} \& {T'^2}  \arrow[rd,"{\otimes '}"] \arrow[rddd,"{\chi}", Rightarrow, shorten <=10ex, shorten >= 10ex]   \& {} \\
	{T'} \arrow[rd,"{\iota \x 1}", Rightarrow, shorten <=3ex, shorten >= 3ex]    \arrow[ru,"{I'\x 1}"] \& \hphantom{.} \& {T'} \\
	{} \& {T^2} \arrow[uu,"{F\x F}"]  \arrow[rd,"{\otimes}"''] \arrow[d,Rightarrow,"{l}",shorten <=1.5ex, shorten >= 1.5ex] \& {} \\
	{T} \arrow[uu,"{F}"] \arrow[rr,"{1}"']  \arrow[ru,"{1\x I}"']\& \hphantom{.} \& {T} \arrow[uu,"{F}"''] 
\end{cdsN}
\h[3]\aM{\gamma}\h[3]
\begin{cdsN}{5.5}{5.5}
	{} \& {T'^2} \arrow[rd,"{\otimes'}"] \arrow[d,"{l}", Rightarrow,shorten <=1.5ex,shorten >=1.5ex] \& {} \\
	{T'}  \arrow[rr,"{1}"', ""{name=O}]  \arrow[ru,"{I' \x 1}"] \& \hphantom{.} \& {T'} \arrow[lldd,"{}", equal, shorten <= 11.3ex, shorten >=11.3 ex]  \\
	{} \& {} \& {} \\
	{T} \arrow[uu,"{F}"] \arrow[rr,"{1}"'] \& {} \& {T} \arrow[uu,"{F}"]
\end{cdsN}
\end{eqD*}
\begin{eqD*}
	\begin{cdsN}{5.5}{5.5}
		{} \& {T^2} \arrow[rd,"{\otimes}"] \arrow[d,"{\hat{r}}", Leftarrow,shorten <=1.5ex,shorten >=1.5ex] \& {} \\
		{T}  \arrow[rr,"{1}"', ""{name=O}]  \arrow[ru,"{1\x I}"] \& \hphantom{.} \& {T} \arrow[lldd,"{}", equal, shorten <= 11.3ex, shorten >=11.3 ex]  \\
		{} \& {} \& {} \\
		{T'} \arrow[uu,"{F}", leftarrow] \arrow[rr,"{1}"'] \& {} \& {T'} \arrow[uu,"{F}", leftarrow]
	\end{cdsN}
	\h[3]\aM{\delta}\h[3]
	\begin{cdsN}{5.5}{5.5}
		{} \& {T^2}  \arrow[rd,"{\otimes}"]   \& {} \\
		{T}   \arrow[ru,"{1 \x I}"] \& \hphantom{.} \& {T} \\
		{} \& {T'^2} \arrow[ru,"{\chi}", Rightarrow, shorten <=3ex, shorten >= 3ex]  \arrow[uu,"{F\x F}", leftarrow]  \arrow[rd,"{\otimes'}"''] \arrow[d,Leftarrow,"{l}",shorten <=1.5ex, shorten >= 1.5ex] \& {} \\
		{T'} \arrow[ruuu,"{1 \x \iota}", Rightarrow, shorten <=10ex, shorten >= 10ex]   \arrow[uu,"{F}", leftarrow] \arrow[rr,"{1}"']  \arrow[ru,"{1\x I'}"']\& \hphantom{.} \& {T'} \arrow[uu,"{F}"'', leftarrow] 
	\end{cdsN}
	\end{eqD*}
	\end{itemize}
See Definition 3.3.1 of \cite{Gurskithesis} for the axioms that these data are required to satisfy.
\end{defne}

\begin{defne}
	Let $F,G\: T \to T'$ be trihomomorphisms. A \dfn{tritransformation} $\theta\: F \Rightarrow G$ is given by the following data:
	\begin{itemize}
		\item given $a\in T$ a morphism $\theta_a\: F(a) \to G(a)$ in $T'$;
		\item given $a,b\in T$ an adjoint equivalence 
		\csq[l][7][7][\theta]{T(a,b)}{T'(F(a),F(b))}{T'(G(a),G(b))}{T'(F(a),G(b));}{F}{G}{T'(1, \theta_b)}{T'(\theta_a,1)}
		\item given $a,b,c\in T$ invertible modifications 
		\begin{eqD*}
			\scalebox{0.6}[0.65]{
		\begin{cdN}
			{} \& {} \& {T(b,c)\x T(a,b)} \arrow[ddl,"{G\x G}"'] \arrow[ddll,"{\otimes}"', bend right=15] \arrow[r,"{F\x F}"] \arrow[d,"{G\x F}"]\& {T'(F(b),F(c))\x T'(F(a),F(b))} \arrow[ld,"{\theta\x 1}", Rightarrow, ,shorten <=9.5ex, shorten >= 11.5ex] \arrow[d,"{T'(1,\theta_c)\x 1}"] \\
			{} \& {} \& {T'(G(b),G(c))\x T'(F(a),F(b))} \arrow[ld,"{1\x \theta}", Rightarrow,shorten <=9.5ex, shorten >= 11ex]\arrow[r,"{T'(\theta_b)\x 1}"] \arrow[d,"{1 \x T'(1,\theta_b)}"]\& {T'(F(b),G(c))\x T'(F(a),F(b))} \arrow[ld,"{\alpha}",shorten <=10.5ex, shorten >= 10.5ex, Rightarrow] \arrow[d,"{\otimes}"] \\
			{T(a,c)} \arrow[r,"{\chi}", Leftarrow,shorten <=1.5ex, shorten >= 1.5ex] \arrow[rrd,"{G}"', bend right=15] \& {T'(G(b),G(c))\x T'(G(a),G(b))} \arrow[r,"{1\x T'(\theta_a,1)}"] \arrow[rd,"{\otimes}"] \& {T'(G(b),G(c))\x T'(F(a),G(b))} \arrow[d,"{\hat{\alpha}}", Rightarrow,shorten <=1.5ex, shorten >= 1.5ex] \arrow[r,"{\otimes}"] \& {T'(F(a), G(c))} \\
			{} \& {} \& {T'(G(a), G(c))} \arrow[ru,"{T'(\theta_a,1)}"'] \& {} 
	\end{cdN}}
		\end{eqD*}
	\begin{eqD*}
		\begin{cdN}
			{\hphantom{.}} \arrow[d,"{\Pi}", triple]\\
			{\hphantom{.}}
		\end{cdN}
	\end{eqD*}
\begin{eqD*}
	\scalebox{0.6}[0.65]{
		\begin{cdN}
			{} \& {} \& {T(b,c)\x T(a,b)}  \arrow[ddll,"{\otimes}"', bend right=15] \arrow[r,"{F\x F}"] \arrow[d,"{F\x F}"]\& {T'(F(b),F(c))\x T'(F(a),F(b))} \arrow[ldd,"{\alpha}", Rightarrow, ,shorten <=15.5ex, shorten >= 15.5ex] \arrow[d,"{T'(1,\theta_c)\x 1}"] \\
			{} \& {} \& {T'(F(b),F(c))\x T'(F(a),F(b))} \arrow[lld,"{\chi}", Rightarrow,shorten <=9.5ex, shorten >= 9.5ex]  \arrow[d,"\otimes"]\& {T'(F(b),G(c))\x T'(F(a),F(b))}  \arrow[d,"{\otimes}"] \\
			{T(a,c)} \arrow[rr,"{F}"] \arrow[rrd,"{G}"', bend right=15] \& {}   \& {T'(F(a),F(c))}  \arrow[d,"{\hat{\alpha}}", Rightarrow,shorten <=1.5ex, shorten >= 1.5ex] \arrow[r,"{T'(1,\theta_c)}"] \& {T'(F(a), G(c))} \\
			{} \& {} \& {T'(G(a), G(c))} \arrow[ru,"{T'(\theta_a,1)}"'] \& {} 
	\end{cdN}}
\end{eqD*}
			and
			
			\begin{eqD*}
				\scalebox{0.7}{
				\begin{cdN}
					{} \& {1}\arrow[ld,"{\iota}", Rightarrow ,shorten <=7.5ex, shorten >= 7.5ex] \arrow[d,"{I_{F(a)}}"] \arrow[ld,"{I_a}"',""{name=B}, bend right= 25] \arrow[rdd,"{\theta_a}",""{name=A}, bend left=35] \& {} \\
					{T(a,a)} \arrow[r,"{F}"] \arrow[d,"{G}"']\& {T'(F(a), F(a))} \arrow[ld,"{\theta}", Rightarrow ,shorten <=5.5ex, shorten >= 5.5ex]  \arrow[from=A,"{\hat{r}}", Rightarrow ,shorten <=2.5ex, shorten >= 2.5ex] \arrow[rd,"{T'(1,\theta_a)}"] \& {} \\
					{T'(G(a), G(a))} \arrow[rr,"{T'(\theta_a,1)}"'] \& {} \& {T'(F(a), G(a))} 
					\end{cdN}
				\h[-40]
				\begin{cdN}
			{\hphantom{.}}  \arrow[r,"{M}", triple]\& {\hphantom{.}}
			\end{cdN}
		\h[-40]
		\begin{cdN}
			{} \& {1} \arrow[ld,"{I_a}"',""{name=B}, bend right= 25] \arrow[rdd,"{\theta_a}",""{name=A}, bend left=35] \arrow[ddl,"{I_{G(a)}}",""{name=C},bend left=35] \& {} \\
			{T(a,a)} \arrow[d,"{G}"]\& {\hphantom{.}} \arrow[l,"{\iota}", Rightarrow ,shorten <=3.5ex, shorten >= 3.5ex]\& {} \\
			{T'(G(a), G(a))} \arrow[from=A,"{\hat{l}}", Rightarrow ,shorten <=10.5ex, shorten >= 10.5ex, shift left= 3ex] \arrow[rr,"{T'(\theta_a,1)}"']\& {} \& {T'(F(a), G(a));} 
			\end{cdN}}
			\end{eqD*}
		% N PBlaxoplax:p-l-o 7 7 twocell sh< sh> pos equal:l-d-r Square
	\end{itemize}
See Definition 3.3.6 of \cite{Gurskithesis} for the axioms that these data are required to satisfy.
\end{defne}

\begin{prop}
	The assignment $\qst{\cX}{\cG}$ of Definition \ref{qp2} is a trihomomorphism.
\end{prop}

\begin{proof}
	Let $\cY,\cZ$ in $\K$. We prove that the assignment 
	$$\qst{\cX}{\cG}_{\cY,\cZ}\: \K\op (\cY,\cZ) \to\twocat(\qst{\cX}{\cG}(\cY), \qst{\cX}{\cG}(\cZ))$$
	that sends a morphism $f\: \cZ \to \cY$ into $\qst{\cX}{\cG}(f)$ and a 2-cell $\Gamma\: f \Rightarrow g$ into $\qst{\cX}{\cG}(\Gamma)$ is a 2-functor. 
	
	Let $f\: \cZ \to \cY$ be a morphism in $\K$. The 2-natural transformation $\qst{\cX}{\cG}(\id{f})$ has component relative to $(\cP, \alpha_{\cP})\in \qst{\cX}{\cG}(\cY)$ induced by the universal property of the iso-comma object $\cP\xp{\cY}{f}\cZ$ as in the following diagram
	\begin{eqD*}
		\commaunivvN{\cP\xp{\cY}{f}\cZ}{\pi_{\cP}\st f}{f\st \pi_{\cP}}{\qst{\cX}{\cG}(\id{f})_{(\cP,\alpha_{\cP})}}{\cP\xp{\cY}{f}\cZ}{\cP}{\cZ}{\cY}{\pi_{\cP}\st f}{f\st \pi_{\cP}}{\pi_{\cP}}{f}{\alpha_{\cP}^{f}}
		\qquad= \quad
		\begin{cdN}
			{\cP\xp{\cY}{f}\cZ} \arrow[d,"{f\st \pi_{\cP}}"'] \arrow[r,"{\pi_{\cP}\st f}"] \&[5ex] {\cP} \arrow[d,"{\pi_{\cP}}"] \arrow[ld,"{\alpha_{\cP}^{f}}"',shift left=-1.3ex,Rightarrow ,shorten <=5.5ex, shorten >= 5.5ex]\\[3ex]
			{\cZ} \arrow[r,"{f}"',""{name=A}] \& {\cY}
		\end{cdN}
	\end{eqD*}
and so $\qst{\cX}{\cG}(\id{f})_{(\cP,\alpha_{\cP})}$ is equal to $(\id{\qst{\cX}{\cG}(f)})_{(\cP,\alpha_{\cP})}$. This shows that $\qst{\cX}{\cG}_{\cY,\cZ}$ preserves identities.

Let now $f,g,h\: \cZ \to \cY$ be morphisms in $\K$ and $\Gamma\: f \Rightarrow g$ and $\Lambda \: g \Rightarrow h$ be 2-cells in $\K$. The 2-natural transformation $\qst{\cX}{\cG}(\Lambda \c \Gamma)$ has component relative to $(\cP, \alpha_{\cP})\in \qst{\cX}{\cG}(\cY)$ induced by the universal property of the iso-comma object $\cP\xp{\cY}{h}\cZ$ as in the following diagram
\begin{eqD*}
	\commaunivvN{\cP\xp{\cY}{f}\cZ}{\pi_{\cP}\st f}{f\st \pi_{\cP}}{\qst{\cX}{\cG}(\Lambda \c \Gamma)_{(\cP,\alpha_{\cP})}}{\cP\xp{\cY}{h}\cZ}{\cP}{\cZ}{\cY}{\pi_{\cP}\st h}{h\st \pi_{\cP}}{\pi_{\cP}}{h}{\alpha_{\cP}^{h}}
	\qquad= \quad
	\begin{cdN}
		{\cP\xp{\cY}{f}\cZ} \arrow[d,"{f\st \pi_{\cP}}"'] \arrow[r,"{\pi_{\cP}\st f}"] \&[5ex] {\cP} \arrow[d,"{\pi_{\cP}}"] \arrow[ld,"{\alpha_{\cP}^{f}}"',shift left=-1.3ex,Rightarrow ,shorten <=5.5ex, shorten >= 5.5ex]\\[3ex]
		{\cZ} \arrow[r,"{f}",bend left=40 ,""'{name=B}] \arrow[r,"{h}"',""{name=A}] \arrow[from=B,to=A,Rightarrow,"{\Lambda \c \Gamma}"]\& {\cY}
	\end{cdN}
\end{eqD*}
and so $\qst{\cX}{\cG}(\Lambda \c \Gamma)_{(\cP,\alpha_{\cP})}$ is equal to $\qst{\cX}{\cG}(\Lambda)_{(\cP,\alpha_{\cP})} \c \qst{\cX}{\cG}(\Gamma)_{(\cP,\alpha_{\cP})}$. This shows that $\qst{\cX}{\cG}_{\cY,\cZ}$ preserves composition.

Let now $\cY\in \K$. We need to prove that there exists an adjoint equivalence
$$\iota\: \id{ \qst{\cX}{\cG}(\cY)}\Rightarrow \qst{\cX}{\cG}( \id{\cY})$$
Given $(\cP,\alpha_{\cP})\in \qst{\cX}{\cG}(\cY)$, we define 
$$\iota_{(\cP,\alpha_{\cP})}\: \id{ \qst{\cX}{\cG}(\cY)}((\cP,\alpha_{\cP})) \to \qst{\cX}{\cG}(\id{\cY})((\cP,\alpha_{\cP}))$$
as an adjoint equivalence between $(\cP, \alpha_{\cP})$ and $(\cP\xp{\cY}{\id{\cY}}\cY, \alpha_{\cP}\c \pi_{\cP}\st \id{\cY})$ given by the fact that they are both bi-iso-comma objects of  $\pi_{\cP}$ and $\id{\cY}$. 
Then, given a morphism $\phi\: (\cP,\alpha_{\cP}) \to (\cQ\alpha_{\cQ})$ in $\qst{\cX}{\cG}(\cY)$, we can induce the invertible 2-cell
\csq[l][7][7][\iota_{\phi}]{\cP}{\cP \xp{\cY}{\id{\cY}}\cY} {\cQ}{\cQ \xp{\cY}{\id{\cY}}\cY} {\iota_{(\cP,\alpha_{\cP})}} {\phi} {\qst{\cX}{\cG}(\id{\cY})(\phi)} {\iota_{(\cQ,\alpha_{\cQ})}}
using the universal property of the iso-comma object $\cQ \xp{\cY}{\id{\cY}}\cY$ thanks to the following compatible isomorphic 2-cells:
\begin{cd}
	{\cP } \arrow[dd,"{\phi}"'] \arrow[rd,equal] \arrow[r,"{\iota_{(\cP,\alpha_{\cP})}}"] \&[5ex] {\cP \xp{\cY}{\id{\cY}}\cZ} \arrow[d,"{\pi_{\cP}\st \id{\cY}}"] \arrow[r,"{\qst{\cX}{\cG}(\id{\cY})(\phi)}"] \&[5ex] {\cQ \xp{\cY}{\id{\cY}}\cZ} \arrow[dd,"{\pi_{\cQ}\st \id{\cY}}"] \\
	{} \& {\cP} \arrow[rd,"{\phi}"] \& {}\\
	{\cQ } \arrow[rr,"{}",equal, bend left=20] \arrow[r,"{\qst{\cX}{\cG}(\id{\cY})(\phi)}"'] \& {\cQ \xp{\cY}{\id{\cY}}\cZ} \arrow[r,"{\pi_{\cQ}\st \id{\cY}}"'] \& {\cQ}
\end{cd}
and 
\begin{cd}
	{\cP }  \arrow[rrdd,"{ \pi_{\cP}}"]\arrow[dd,"{\phi}"'] \arrow[r,"{\iota_{(\cP,\alpha_{\cP})}}" ,""{name=F}] \&[5ex] {\cP \xp{\cY}{\id{\cY}}\cZ} \arrow[rdd,"{\id{\cY}\st \pi_{\cP}}"] \arrow[r,"{\qst{\cX}{\cG}(\id{\cY})(\phi)}"] \&[5ex] {\cQ \xp{\cY}{\id{\cY}}\cZ} \arrow[dd,"{\id{\cY}\st \pi_{\cQ}}"] \\
	{} \& {} \& {}\\
	{\cQ } \arrow[rr,"\pi_{\cQ}", bend left=20] \arrow[r,"{\qst{\cX}{\cG}(\id{\cY})(\phi)}"',""{name=G,pos=0.2}] \arrow[from=F,to=G,"{\gamma_{\phi}^{-1}}",Rightarrow ,shorten <=9.5ex, shorten >= 6.5ex]\& {\cQ \xp{\cY}{\id{\cY}}\cZ} \arrow[r,"{\id{\cY}\st \pi_{\cQ}}"']\& {\cY}
\end{cd}	
Moreover, it is straightforward to prove that the assignment of the structure 2-cells $\iota_\phi$ is pseudofunctorial and that $\iota$ is an adjoint equivalence.

Let now $\cY,\cZ,\cT\in \K$. We need to prove that there exists an adjoint equivalence
$$\chi\: \otimes' \c (\qst{\cX}{\cG} \x \qst{\cX}{\cG}) \Rightarrow \qst{\cX}{\cG} \c \otimes$$
Given $g\: \cT \to \cZ$ and $f\: \cZ \to \cY$, we need to define a 2-natural transformation
$$\chi_{f,g}\:  \qst{\cX}{\cG}(g) \c \qst{\cX}{\cG}(f) \Rightarrow \qst{\cX}{\cG}(f\c g)$$
Given $(\cP, \alpha_{\cP})\in \qst{\cX}{\cG}(\cY)$, we define
$$(\chi_{f,g})_{(\cP, \alpha_{\cP})}\:  (\qst{\cX}{\cG}(g) \c \qst{\cX}{\cG}(f))((\cP, \alpha_{\cP})) \Rightarrow \qst{\cX}{\cG}(f\c g) ((\cP, \alpha_{\cP}))$$
as an adjoint equivalence between $(\cP \x[\cY]\cZ)\x[\cZ] \cT$ and $\cP\x[\cY]\cT$ given by the fact that they are both iso-comma objects of $\pi_{\cP}$ and $f\c g$. 

Given a morphism $\phi\: (\cP,\alpha_{\cP}) \to (\cQ\alpha_{\cQ})$ in $\qst{\cX}{\cG}(\cY)$, the following square is then commutative
\csq{(\cP\x[\cY]\cZ)\x[\cZ]\cT}{\cP \xp{\cY}{}\cT} {(\cQ\x[\cY]\cZ)\x[\cZ]\cT}{\cQ \xp{\cY}{}\cT} {(\chi_{f,g})_{(\cP,\alpha_{\cP})}} {(\qst{\cX}{\cG}(g) \c \qst{\cX}{\cG}(f))(\phi)} {\qst{\cX}{\cG}(f\c g)(\phi)} {(\chi_{f,g})_{(\cQ,\alpha_{\cQ})}}
thanks to the following commutative diagrams:
\begin{cd}
	{(\cP \x[\cY]\cZ)\x[\cZ]\cT} \arrow[rd,"{(f\st \pi_{\cP})\st g}"] \arrow[ddd,"{(\qst{\cX}{\cG}(f) \c \qst{\cX}{\cG}(g))(\phi)}"']\arrow[rr,"{(\chi_{f,g})_{\cP}}"]\& {} \& {\cP \x[\cY]\cT} \arrow[d,"{\pi\st_{\cP}(f\c g)}"]  \arrow[r,"{ \qst{\cX}{\cG}(f\c g)(\phi)}"]\&[4ex] {\cQ \x[\cY]\cT} \arrow[ddd,"{\pi_{\cQ}\st (f\c g)}"] \\
	{} \& {\cP \x[\cY]\cZ} \arrow[r,"{\pi\st_{\cP}}"]  \arrow[d,"{\qst{\cX}{\cG}(f)(\phi)}"{description}]\& {\cP} \arrow[rdd,"{\phi}"] \& {} \\
	{} \& {\cQ \x[\cY]\cZ} \arrow[rrd,"{\pi\st_{\cQ}}"] \& {} \& {} \\
	{(\cQ \x[\cY]\cZ)\x[\cZ]\cT} \arrow[ru,"{(f\st \pi_{\cQ})\st g}"] \arrow[rr,"{(\chi_{f,g})_{\cQ}}"'] \& {} \& {\cQ \x[\cY]\cT} \arrow[r,"{\pi_{\cQ}\st (f\c g)}"']  \& {\cQ} 
\end{cd}
and 
\begin{cd}
	{(\cP \x[\cY]\cZ)\x[\cZ]\cT} \arrow[rrrd,"{g\st (f\st \pi_{\cP})}"]  \arrow[d,"{(\qst{\cX}{\cG}(f) \c \qst{\cX}{\cG}(g))(\phi)}"']\arrow[rr,"{(\chi_{f,g})_{\cP}}"]\& {} \& {\cP \x[\cY]\cT} \arrow[rd,"{(f\c g)\st \pi_{\cP}}"] \arrow[r,"{ \qst{\cX}{\cG}(f\c g)(\phi)}"]\&[4ex] {\cQ \x[\cY]\cT} \arrow[d,"{(f\c g)\st \pi_{\cQ}}"] \\[7ex]
	{(\cQ \x[\cY]\cZ)\x[\cZ]\cT} \arrow[rrr,"{g\st (f\st \pi_{\cQ})}", bend left=13] \arrow[rr,"{(\chi_{f,g})_{\cQ}}"'] \& {} \& {\cQ \x[\cY]\cT} \arrow[r,"{(f\c g)\st \pi_{\cQ}}"']  \& {\cT} 
\end{cd}

This shows that $\chi_{f,g}$ is 2-natural. Moreover, when $f$ and $g$ vary the $\chi_{f,g}$'s form a 2-natural transformation. Indeed, given morphisms $g,g'\: \cT\to \cZ$ and $f,f'\: \cZ\to \cY$ and 2-cells $\Gamma\: g \Rightarrow g'$ and $\Lambda\: f \Rightarrow f'$ the following equality holds
$$\qst{\cX}{\cG}(\Lambda \star \Gamma) \c \chi_{f,g} = \chi_{f',g'} \c (\qst{\cX}{\cG}(\Lambda) \star \qst{\cX}{\cG}(\Gamma).$$
This is because, for every $(\cP, \alpha_{\cP})\in \qst{\cX}{\cG}(\cY)$, the following diagram is commutative
\begin{cd}
	{(\cP\xp{\cY}{f}\cZ) \xp{\cZ}{g}\cT} \arrow[rr,"{(\chi_{f,g})_{\cP}}"] \arrow[d,"{\qst{\cX}{\cG}(g)((\qst{\cX}{\cG}(\Lambda))_{\cP})}"']\&[6ex] {} \&[3ex] {\cP \xp{\cY}{f\c g}\cT} \arrow[d,"{(\qst{\cX}{\cG}(\Lambda \star \Gamma))_{\cP}}"] \\[5ex]
	{(\cP\xp{\cY}{f'}\cZ) \xp{\cZ}{g}\cT} \arrow[r,"{(\qst{\cX}{\cG}(\Gamma))_{\cP \xp{\cY}{f'}}}"']\& {(\cP\xp{\cY}{f'}\cZ) \xp{\cZ}{g'}\cT} \arrow[r,"{\chi_{f',g'}}"'] \& {\cP \xp{\cY}{f'\c g'}\cT} 
\end{cd}
thanks to the commutativity of 
\begin{cd}
	{(\cP\xp{\cY}{f}\cZ) \xp{\cZ}{g}\cT} \arrow[rd,"{(f\st \pi\st_{\cP})\st g}"{description}] \arrow[r,"{(\chi_{f,g})_{\cP}}"] \arrow[d,"{\qst{\cX}{\cG}(g)((\qst{\cX}{\cG}(\Lambda))_{\cP})}"']\&[4ex] {\cP \xp{\cY}{f\c g}\cT} \arrow[rddd,"{\pi\st_{\cP}(f\c g)}", bend left=15] \arrow[r,"{(\qst{\cX}{\cG}(\Lambda \star \Gamma))_{\cP}}"] \&[4ex] {\cP \xp{\cY}{f'\c g'}\cT} \arrow[ddd,"{\pi\st_{\cP}(f'\c g')}"]  \\[3ex]
	{(\cP\xp{\cY}{f'}\cZ) \xp{\cZ}{g}\cT} \arrow[rd,"{((f')\st\pi_{\cP})\st g}"{description}] \arrow[dd,"{(\qst{\cX}{\cG}(\Gamma))_{\cP \xp{\cY}{f'}}}"'] \& {\cP\xp{\cY}{f}\cZ} \arrow[rdd,"{\pi\st_{\cP}f}", bend left=15] \arrow[d,"{\qst{\cX}{\cG}(\Lambda)_{\cP}}"{description}] \& {} \\
	{} \& {\cP\xp{\cY}{f'}\cZ} \arrow[rd,"{\pi\st_{\cP}f'}"]\& {} \\
	{(\cP\xp{\cY}{f'}\cZ) \xp{\cZ}{g'}\cT} \arrow[ru,"{((f')\st \pi_{\cP})\st g'}"] \arrow[r,"{\chi_{f',g'}}"'] \& {\cP \xp{\cY}{f'\c g'}\cT} \arrow[r,"{\pi\st_{\cP}(f'\c g')}"'] \& {\cP} 
\end{cd}
and  
\begin{cd}
	{(\cP\xp{\cY}{f}\cZ) \xp{\cZ}{g}\cT} \arrow[rrdd,"{g\st (f\st \pi_{\cP})}"{description}, bend left=10] \arrow[r,"{(\chi_{f,g})_{\cP}}"] \arrow[d,"{\qst{\cX}{\cG}(g)((\qst{\cX}{\cG}(\Lambda))_{\cP})}"']\&[4ex] {\cP \xp{\cY}{f\c g}\cT} \arrow[rdd,"{(f\c g)\st \pi_{\cP}}"{description}, bend left=10] \arrow[r,"{(\qst{\cX}{\cG}(\Lambda \star \Gamma))_{\cP}}"] \&[4ex] {\cP \xp{\cY}{f'\c g'}\cT} \arrow[dd,"{\pi\st_{\cP}(f'\c g')}"] \\
	{(\cP\xp{\cY}{f'}\cZ) \xp{\cZ}{g}\cT} \arrow[rrd,"{g\st ((f')\st \pi_{\cP}}"{description}, bend left=10] \arrow[d,"{(\qst{\cX}{\cG}(\Gamma))_{\cP \xp{\cY}{f'}}}"'] \& {} \& {} \\
	{(\cP\xp{\cY}{f'}\cZ) \xp{\cZ}{g'}\cT} \arrow[rr,"{(g')\st((f')\st\pi_{\cP})}", bend left=10]\arrow[r,"{\chi_{f',g'}}"'] \& {\cP \xp{\cY}{f'\c g'}\cT} \arrow[r,"{\pi\st_{\cP}(f'\c g')}"'] \& {\cT} 
\end{cd}
This shows that $\chi$ is an adjoint equivalence. 

Consider now a chain of morphisms 
$$\cW \ar{h} \cT \ar{g} \cZ \ar{f} \cY$$
in $\K$. The following equality of 2-natural transformation holds
$$\chi_{f,g\c h} \c (\chi_{g,h} \star \qst{\cX}{\cG} (f)) = \chi_{f\c g,h} \c (\qst{\cX}{\cG}(h) \star \chi_{f,g})$$
because it holds on components thanks to the compatibility between the iso-comma objects along the composite of three morphisms and the pasting of the three iso-comma objects. This shows that the invertible modification $\omega$ required in Definition \ref{trihomomorphism} is the identity for $\qst{\cX}{\cG}$.

Moreover, the invertible modification $\gamma$ of Definition \ref{trihomomorphism}, is the identity as well. Indeed, given a morphism $\cZ\ar{f}\cY$ in $\K$, the pseudonatural transformations $\chi_{f,\id{\cZ}} \c (\iota \star  \qst{\cX}{\cG}(f))$ and $\id{ \qst{\cX}{\cG}(f)}$ have the same components by the universal properties of the iso-comma object and the same structure 2-cells thanks to the fact that the structure 2-cells of $\iota$ of the form $\iota_{ \qst{\cX}{\cG}(f)(\phi)}$ with $\phi$ a morphism in $\qst{\cX}{\cG}(\cY)$ are identities by definition of $\qst{\cX}{\cG}(f)$. Analogously, the invertible modification $\delta$ of Definition \ref{trihomomorphism} is the identity as well.

Finally, the axioms of trihomomorphisms are trivially satisfied by $\qst{\cX}{\cG}$ since all the modifications involved are identities.
\end{proof}

Analogously to the one-dimensional case, we can consider the particular case of $[\cT/\cG]$ where $\cT$ is the terminal object of $\K$ and so the  morphisms with target $\cT$ are uniquely determined and the  2-category $\qst{\cT}{\cG}(\cY)$ is isomorphic to $\twoBun{\cG}{\cY}$. 

\begin{defne} \label{class2}
	 The quotient pre-2-stack $\qst{\cT}{\cG}$ is called \dfn{classifying pre-2-stack} and will be denoted $2\clst{G}$.
\end{defne}

In the remaining part of this section we will prove that, if the $(2,1)$-category $\K$ is bicocomplete and such that iso-comma objects preserve bicolimits and the bitopology $\tau$ is subcanonical, the quotient pre-2-stacks  of Definition \ref{qp2} are 2-stacks in the sense of \cite{2stacks}. To prove this important result, we will use the following explicit characterization of 2-stack that we prove in \cite{2stacks}.

\begin{teor}[\cite{2stacks}, Theorem 3.18 ]\label{char2stacks}
	Let $(\K,\tau)$ be a bisite. A trihomomorphism $F\: \K\op \to \Bicat$ is a 2-stack if and only if for every $C\in \K$ and every covering bisieve $S\in \tau(C)$ the following conditions are satisfied:
	\begin{itemize}
		\item [(O)] every weak descent datum for $S$ of elements of $F$ is weakly effective;
		\item [(M)] every descent datum for $S$ of morphisms of $F$ is effective;
		\item [(2C)] every matching family for $S$ of 2-cells of $F$ has a unique amalgamation.
	\end{itemize}
\end{teor}

 Moreover, we will need the crucial result that every object of a subcanonical bisite is the sigma-bicolimit of each covering bisieve over it. We prove this result in \cite{2stacks}.

\begin{teor}[\cite{2stacks},Theorem 2.24]\label{teorsigmabicolimbisieves}
	Let $\tau$ be a subcanonical bitopology and let $S\: R \Rightarrow y(C)$ be a covering bisieve over $C$. Then $S$ is a sigma-bicolim bisieve, that is 
$$C=\sigmabicolim{F}$$
where $F\: \Groth{R}\to \K$ is the 2-functor of  projection to the first component.
\end{teor}

Another important result proven in \cite{2stacks} that will be helpful for the proof is the following.

\begin{prop}[\cite{2stacks},Proposition 2.26]\label{coconofstar}
	Let $\K$ be a small 2-category with iso-comma objects and let $\tau$ be a subcanonical bitopology on it. Let then $S\: R \Rightarrow \K(-,Y)$ be a covering bisieve over $Y\in\K$ and $f\: X\to Y$ be a morphism in $\K$. Then $X$ is the sigma-bicolimit of the 2-functor 
	$$F\:  \Grothdiag{R} \ar{\opn{inc}} \laxslice{\K}{Y} \ar{f\st} \laxslice{\K}{X}  \ar{\opn{dom}} \K$$
	where $\opn{inc}=\Groth{S}\: \Groth{R} \to \Groth{\K(-,Y)}$ is the inclusion 2-functor  and $f\st$ is the 2-functor of iso-comma object along the morphism $f$. 
\end{prop}

Finally, we will need to apply the results about colimits in lax slices proved by Mesiti in \cite{Mesiticolimits} and more in detail in the case of sigma-bicolimits in his PhD thesis \cite{Mesitithesis}.

\begin{teor}\label{aretwostacks}
	Let $\K$ be a bicocomplete $(2,1)$-category with finite flexible limits and such that iso-comma objects preserve bicolimits. Let then $\tau$ be a subcanonical bitopology on $\K$ and $\cX,\cG\in \K$  with $\cG$ an internal 2-group. Then the quotient pre-2-stack $\qst{\cX}{\cG}$ is a 2-stack.
\end{teor}

\begin{oss}
	The hypothesis that iso-comma objects preserve bicolimits in $\K$ means that for every morphism $a\: \cT \to \cW$ in $\K$ the 2-functor 
	$$a\st\: \laxslice{\K}{\cW} \to \laxslice{\K}{\cT}$$
	of iso-comma along  the morphism $a$ preserves bicolimits. This hypothesis will be crucial in our proof of Theorem \ref{aretwostacks}.
\end{oss}

\begin{proof}
Let $\cY\in \K$ and let $S\: R \Rightarrow \K(-,\cY)$ be a covering bisieve over $\cY$. Thanks to Theorem \ref{char2stacks}, it suffices to prove that the following conditions hold: 
	\begin{itemize}
	\item [(O)] every weak descent datum for $S$ of elements of $F$ is weakly effective;
	\item [(M)] every descent datum for $S$ of morphisms of $F$ is effective;
	\item [(2C)] every matching family for $S$ of 2-cells of $F$ has a unique amalgamation.
\end{itemize}
We start proving (2C). 

Let $\sigma,\rho\: (\cP, \alpha_{\cP}) \to (\cQ, \alpha_{\cQ})$ be morphisms in $\qst{\cX}{\cG}(\cY)$. Consider a matching family for $S$ of 2-cells of $\qst{\cX}{\cG}$ that assigns to each morphism $f\: \cZ\to \cY\in S$ a 2-cell 
$$w_f\: \qst{\cX}{\cG}(f)(\sigma) \Rightarrow \qst{\cX}{\cG}(f)(\rho)$$
We need to prove that there exists a unique 2-cell
$$w\: \sigma \Rightarrow \rho$$
such that $\qst{\cX}{\cG}(f)(w)=w_f$ for every  $f\: \cZ\to \cY\in S$.

By Proposition \ref{coconofstar}, $\cP$ is the sigma-bicolimit of the 2-functor
$$F\:  \Grothdiag{R} \ar{\opn{inc}} \laxslice{\K}{\cY} \ar{\pi\st_{\cP}} \laxslice{\K}{\cP}  \ar{\opn{dom}} \K$$
with universal sigma-bicocone given by the identical 2-cells of the form
\begin{eqD}{cocuniv}
	\begin{cdN}
		{\cP\xp{\cY}{h}\cW} \arrow[rd,"{\pi\st_{\cP} h}",""{name=A}] \arrow[d,"{F((\id{},\alpha))}"',] \&[10ex] {} \\
		{\cP\xp{\cY}{\wt{l\c g}}\cW}   \arrow[r,"{\pi\st_{\cP}(\wt{l\c g})}"'{pos=0.4}] \arrow[d,"{F((g,\id{}))}"',""{name=C}]\& {\cP} \\
		{\cP\xp{\cY}{l}\cT} \arrow[ru,"{\pi\st_{\cP}l}"',""{name=D}] \& {} 
	\end{cdN}
\end{eqD}
for every morphism $(g,\alpha)\: (\cW, \cW \ar{h} \cY) \to (\cT, \cT \ar{l} \cY)$ in $\Groth{R}$.

For every $f\: \cZ\to \cY\in S$, we consider the 2-cell
\begin{eqD}{duecella}
	\begin{cdN}
		{} \&[5ex] {\cP}  \arrow[rd,"{\sigma}", bend left=35] \&[5ex] {} \\
		{\cP \xp{\cY}{f}\cZ} \arrow[r,"{\qst{\cX}{\cG}(f)(\sigma)}", ,""{name=D}, bend left=35] \arrow[r,"{\qst{\cX}{\cG}(f)(\rho)}"', ,""{name=E}, bend right=35] \arrow[ru,"{\pi\st_{\cP} f}", bend left=35] \arrow[rd,"{\pi\st_{\cP} f}"', bend right=35] \& {\cQ \xp{\cY}{f}\cZ} \arrow[r,"{\pi\st_{\cQ} f}"'] \& {\cQ} \arrow[from=D,to=E,"{w_f}",Rightarrow ,shorten <=1.5ex, shorten >= 1.5ex] \\
		{} \& {\cP}  \arrow[ru,"{\rho}"', bend right=35]\& {}
	\end{cdN}
\end{eqD}
Using the fact that the $w_f$'s are the assignment of a matching family, it is straightforward to prove that the 2-cells of the form (\ref{duecella}) are compatible. In particular, one needs to use the compatibility of the matching family with respect to composition to deal with the morphisms of  $\Groth{R}$ of type $(g,\id{})$ and the compatibility of the matching family with respect to the 2-cells in $S$ to deal with the morphisms of  $\Groth{R}$ of type $(\id{},\alpha)$.
So, by the  two-dimensional universal property of the sigma-bicolimit of $F$, there exists a unique 2-cell $w\: \sigma \Rightarrow \rho$ such that for every $f\: \cZ \to \cY\in S$, the whiskering
\begin{cd}
	{\cP \xp{\cY}{f}\cZ} \arrow[r,"{\pi\st_{\cP} f}"]\& {\cP} \arrow[r,"{\sigma}", bend left=30 ,""'{name=A}] \arrow[r,"{\rho}"', bend right=30 ,""{name=B}] \arrow[from=A,to=B,"{w}", Rightarrow ] \& {\cQ} 
\end{cd}
is equal to the 2-cell of diagram (\ref{duecella}). We now show that $w$ is the desired amalgamation. By definition of $\qst{\cX}{\cG}$, the 2-cell $\qst{\cX}{\cG}(f)(w)$ is the unique that has whiskering with $f\st \pi_{\cQ}$ equal to the identical 2-cell and whiskering with $\pi\st_{\cQ} f$ equal to the 2-cell
\begin{eqD}{F}
\begin{cdN}
	{\cP\xp{\cY}{f}\cZ} \arrow[rd,"{\pi\st_{\cP}}"] \arrow[rr,"{\qst{\cX}{\cG}(f)(\sigma)}"] \arrow[dd,"{\qst{\cX}{\cG}(f)(\rho)}"'] \& {}  \& {\cQ\xp{\cY}{f}\cZ} \arrow[dd,"{\pi\st_{\cQ}}"] \\
	{} \& {\cP} \arrow[rd,"{\sigma}" ,""'{name=A},bend left=30] \arrow[rd,"{\rho}"' ,""{name=B},bend right=30] \arrow[from=A,to=B,"{w}",Rightarrow ] \& {}\\
	{\cQ\xp{\cY}{f}\cZ} \arrow[rr,"{\pi\st_{\cQ}}"']\& {} \&  {\cQ} 
\end{cdN}
\end{eqD}
But the whiskering of $w_f$ with $f\st \pi_{\cQ}$ is equal to the identical 2-cell because $w_f$ is a 2-cell in $\qst{\cX}{\cG}(\cZ)$ and the 2-cells $\gamma_{ \qst{\cX}{\cG}(f)(\sigma)}$ and $\gamma_{ \qst{\cX}{\cG}(f)(\rho)}$ are identities by definition of $\qst{\cX}{\cG}$. Moreover, the whiskering of $w_f$ with $\pi\st_{\cQ}f$, is equal to the 2-cell of diagram (\ref{F}) by construction of $w$. This shows that $\qst{\cX}{\cG}(f)(w)=w_f$ and concludes the proof of the condition (2C).

We now prove the condition $(M)$.

Let $(\cP, \alpha_{\cP}), (\cQ, \alpha_{\cQ})\in \qst{\cX}{\cG}(\cY)$. Consider a descent datum for $S$ of morphisms of $\qst{\cX}{\cG}$ that assigns to every $\cZ \ar{f} \cY\in S$ a morphism 
$$w_f\: \cP\xp{\cY}{f}{\cZ}\to \cQ\xp{\cY}{f}{\cZ}$$
and to every 2-cell $\Gamma\: f \Rightarrow f'$ with $f,f'\: \cZ \to \cY\in S$  an (invertible) 2-cell
	\csq[l][7][7][\eta_{\Gamma}]{\cP\xp{\cY}{f}\cZ} {\cQ\xp{\cY}{f}\cZ}{\cP\xp{\cY}{f'}\cZ} {\cQ\xp{\cY}{f'}\cZ}{w_{f}}{\qst{\cX}{\cG}(\Gamma)_{\cP}}{\qst{\cX}{\cG}(\Gamma)_{\cQ}}{w_{f'}}
Moreover, for every chain $\cZ' \ar{g} \cZ \ar{f} \cY$ with $f\in S$, we have an (invertible) 2-cell
\begin{cd}
	{(\cP\xp{\cY}{f}\cZ)\xp{\cZ}{g}\cZ'} \arrow[r,"{g\st w_f}"]  \arrow[d,"{}",aeq] \&[7ex] {(\cQ\xp{\cY}{f}\cZ)\xp{\cZ}{g}\cZ'} \arrow[d,"{}",aeq]  \arrow[ldd,"{\phi^{f,g}}",Rightarrow ,shorten <=6.5ex, shorten >= 6.5ex]\\[-3ex]
	{\cP \xp{\cY}{f\c g}\cZ'} \arrow[d,"{}",aiso]\& {\cQ \xp{\cY}{f\c g}\cZ'} \arrow[d,"{}",aiso] \\[-3ex]
	{\cP \xp{\cY}{\wt{f\c g}}\cZ'} \arrow[r,"{w_{f\c g}}"'] \&  {\cQ \xp{\cY}{\wt{f\c g}}\cZ'}
\end{cd}
We need to prove that there exist a morphism 
$$w\: \cP \to \cQ$$
in $\qst{\cX}{\cG}(\cY)$ and for every morphism $\cZ \ar{f} \cY\in S$ an (invertible) 2-cell
$$\psi^f\: w_f \Rightarrow \qst{\cX}{\cG}(f)(w)$$
that satisfy the conditions required by the definition of effective descent datum on morphisms (see Definition 3.9 of \cite{2stacks}).
The idea is to induce the morphism $w$ using the universal property of $\cP$ seen as a sigma-bicolimit of a diagram of shape $\Groth{R}$. This is similar to what we did to prove the condition (2C).

As already observed when proving (2C), by Proposition \ref{coconofstar}, $\cP$ is the sigma-bicolimit of the 2-functor
$$F\:  \Grothdiag{R} \ar{\opn{inc}} \laxslice{\K}{\cY} \ar{\pi\st_{\cP}} \laxslice{\K}{\cP}  \ar{\opn{dom}} \K$$
with universal sigma-bicocone as shown in diagram (\ref{cocuniv}).

We construct a sigma-bicocone $\Xi$ of shape $\Groth{R}$ over $\cQ$. Given an object $(\cZ, \cZ\ar{f}\cY)$, we define $\Xi_f$ as the composite 
$$\cP\xp{\cY}{f}\cZ \ar{w_f} \cQ\xp{\cY}{f}\cZ \ar{\pi\st_{\cQ}} \cQ$$ 
 Given a morphism $(g,\id{})\: (\cZ', \wt{f\c g}) \to (\cZ,f)$ in $\Groth{R}$, we define $\Xi_{(g,\id{})}$ as the following 2-cell 
\begin{cd}
	{\cP \xp{\cY}{\wt{f\c g}}\cZ'} \arrow[dd,"{F((g,\id{}))}"']\arrow[r,"{w_{\wt{f\c g}}}"] \& {\cQ \xp{\cY}{\wt{f\c g}}\cZ'} \arrow[dd,"{\widehat{g}_{\cQ}}"] \arrow[rd,"{\pi\st_{\cQ} (\wt{f\c g})}"] \arrow[ldd,"{\kappa_g}",Rightarrow ,shorten <=5.5ex, shorten >= 5.5ex]\&[5ex] {} \\[-3ex]
	{} \& {} \& {\cQ} \\[-3ex]
	{\cP \xp{\cY}{f}\cZ} \arrow[r,"{w_{f}}"]\& {\cQ \xp{\cY}{f}\cZ} \arrow[ru,"{\pi\st_{\cQ} f}"'] \& {} 
\end{cd}
where $\widehat{g}_{\cQ}$ is the morphism induced by the universal property of ${\cQ \xp{\cY}{f}\cZ}$ using the 2-cell 
\begin{cd}
	{\cQ \xp{\cY}{\wt{f\c g}}\cZ} \arrow[rr,"{\pi\st_{\cQ}(\wt{f\c g})}"]  \arrow[d,"{(\wt{f\c g})\st \pi_{\cQ}}"']\& {} \& {\cQ} \arrow[lld,"{\lambda^{\wt{f\c g}}_{\cQ}}"'{inner sep=0.1ex}, Rightarrow ,shorten <=8.5ex, shorten >= 8.5ex, shift left=-2ex] \arrow[d,"{\pi_{\cQ}}"]\\[5ex]
	{\cZ'} \arrow[r,"{g}"'] \arrow[rr,"{\wt{f\c g}}", bend left=35,""'{name=A}] \& {\cZ} \arrow[r,"{f}"'] \arrow[from=A,"{\sigma_{f,g}}"{pos=0.8}, Rightarrow ,]\&  {\cY} 
\end{cd}
and the 2-cell $\kappa_g$ is induced by the two-dimensional universal property of ${\cQ \xp{\cY}{f}\cZ}$ using the compatible 2-cells
\begin{cd}
	{\cP \xp{\cY}{\wt{f\c g}} \cZ'}  \arrow[rrr,"{F((g,\id{}))}"] \arrow[rd,"{}",aiso]  \arrow[ddddd,"{w_{\wt{f\c g}}}"']\&[-5ex] {} \&[-5ex] {} \&[-3ex] {\cP \xp{\cY}{f}\cZ} \arrow[r,"{w_f}"] \& {\cQ \xp{\cY}{f}\cZ} \arrow[ddddd,"{{\pi\st_{\cQ}}f}"] \\[-5ex]
	{} \& {\cP \xp{\cY}{{f\c g}} \cZ'} \arrow[rd,"{}",aeq]  \& {} 	\& {} \& {} \\[-5ex]
	{} \& {} \& {{(\cP \xp{\cY}{{f}} \cZ)\xp{\cZ}{g} \cZ'} } \arrow[llddd,"{\widehat{\phi^{f,g}}}", Rightarrow ,shorten <=8.5ex, shorten >= 8.5ex, shift left=-6ex] \arrow[d,"{g\st w_f}"]	\arrow[ruu,"{}"] \& {} \& {} \\[-2ex]
	{} \& {} \& {{(\cQ\xp{\cY}{{f}} \cZ)\xp{\cZ}{g} \cZ'} } \arrow[rruuu,"{}", bend right=15]	\& {} \& {} \\[-5ex]
	{} \& {\cQ \xp{\cY}{{f\c g}} \cZ'} \arrow[ru,"{}",aeq]  \arrow[rrrd,"{}", bend left=25] \& {} 	\& {} \& {} \\[-5ex]
	{\cQ \xp{\cY}{\wt{f\c g}}\cZ'} \arrow[ru,"{}",aiso] \arrow[rrrr,"{}", bend left=15] \arrow[rrr,"{\widehat{\cQ}_g}"']\& {} \& {} 	\& {\cQ \xp{\cY}{f}\cZ} \arrow[r,"{\pi\st_{\cQ} f}"']  \& {\cQ} 
\end{cd}
and 
\begin{cd}
	{\cP \xp{\cY}{\wt{f\c g}} \cZ'}  \arrow[rrr,"{F((g,\id{}))}"] \arrow[rd,"{}",aiso]  \arrow[ddddd,"{w_{\wt{f\c g}}}"']\&[-5ex] {} \&[-5ex] {} \&[-3ex] {\cP \xp{\cY}{f}\cZ} \arrow[rddddd,"{f\st \pi_{\cP}}",""{name=P}] \arrow[r,"{w_f}"] \arrow[dddd,"{\lambda^{g}_{\cP \xp{\cY}{f}\cZ}}", Rightarrow ,shorten <=7.5ex, shorten >= 7.5ex] \& {\cQ \xp{\cY}{f}\cZ} \arrow[to=P,"{\gamma_{w_f}}", Rightarrow ,shorten <=3.5ex, shorten >= 3.5ex] \arrow[ddddd,"{{f\st \pi_{\cQ}}}"] \\[-5ex]
	{} \& {\cP \xp{\cY}{{f\c g}} \cZ'} \arrow[rd,"{}",aeq]  \& {} 	\& {} \& {} \\[-5ex]
	{} \& {} \& {{(\cP \xp{\cY}{{f}} \cZ)\xp{\cZ}{g} \cZ'} } \arrow[rdd,"{}", bend left=16] \arrow[ruu,"{}"] \arrow[llddd,"{\widehat{\phi^{f,g}}}", Rightarrow ,shorten <=8.5ex, shorten >= 8.5ex, shift left=-6ex] \arrow[d,"{g\st w_f}"] \& {} \& {} \\[-2ex]
	{} \& {} \& {{(\cQ\xp{\cY}{{f}} \cZ)\xp{\cZ}{g} \cZ'} }	 \arrow[rd,"{}"]\& {} \& {} \\[-5ex]
	{} \& {\cQ \xp{\cY}{{f\c g}} \cZ'}  \arrow[rr,"{}"]\arrow[ru,"{}",aeq]   \& {} 	\& {\cZ'}  \arrow[rd,"{g}"']\& {} \\[-5ex]
	{\cQ \xp{\cY}{\wt{f\c g}}\cZ'} \arrow[rrru,"{}", bend right=5] \arrow[ru,"{}",aiso] \arrow[rrr,"{\widehat{g}_{\cQ}}"']\& {} \& {} 	\& {\cQ \xp{\cY}{f}\cZ} \arrow[r,"{\pi\st_{\cQ} f}"']  \& {\cZ} 
\end{cd}
where the unnamed arrows are the projections and the 2-cell $\widehat{\phi^{f,g}}$ is defined by whiskering $\phi^{f,g}$ with the isomorphic 2-cells of pseudoinvertibility of the involved equivalences.

Given  a morphism $(\id{}, \Gamma)\: (\cZ, \cZ \ar{f} \cY) \to (\cZ, \cZ \ar{f'} \cY)$ in $\Groth{R}$, we define $\Xi_{(\id{}, \Gamma)}$ as the following 2-cell
\begin{cd}
	{\cP \xp{\cY}{f}\cZ} \arrow[dd,"{F((\id{},\Gamma))}"'] \arrow[r,"{w_{f}}"] \& {\cQ \xp{\cY}{f}\cZ} \arrow[dd,"{\qst{\cX}{\cG}(\Gamma)_{\cQ} }"] \arrow[rd,"{\pi\st_{\cQ}f}"] \arrow[ldd,"{\eta_{\Gamma}}",Rightarrow ,shorten <=5.5ex, shorten >= 5.5ex]\&[5ex] {} \\[-3ex]
	{} \& {} \& {\cQ} \\[-3ex]
	{\cP \xp{\cY}{f'}\cZ} \arrow[r,"{w_{f'}}"]\& {\cQ \xp{\cY}{f'}\cZ} \arrow[ru,"{\pi\st_{\cQ} f'}"'] \& {} 
\end{cd}
One can then construct the structure 2-cell of the form $\Xi_{(g,\alpha)}$ using the factorization of $(g,\alpha)$ as a composite of a morphism of type $(g,\id{})$ with one of type $(\id{},\alpha)$ and use the compatibility conditions satisfied by the descent datum to verify that these assignment defines a sigma-bicocone of shape $\Groth{R}$ over $\cQ$. This implies that there exists a unique morphism $w\: \cP \to \cQ$ together (invertible) 2-cells 
\begin{cd}
	{} \& {\cP}  \arrow[dd,"{\theta_f}",twoiso ,shorten <=2.5ex, shorten >= 2.5ex] \arrow[rd,"{w}", bend left=30] \& {} \\[-4ex]
	{\cP \xp{\cY}{f}\cZ}  \arrow[ru,"{\pi\st_{\cP} f}", bend left=30] \arrow[rd,"{w_f}"', bend right=30] \&  \& {\cQ} \\[-4ex]
	{} \& {\cQ \xp{\cY}{f}\cZ}  \arrow[ru,"{\pi\st_{\cQ} f}"', bend right=30]\& {}
\end{cd}
for every $\cZ \ar{f} \cY$ such that, given a morphism $(g,\alpha)\: (\cZ, \cZ \ar{f} \cY) \to (\cZ', \cZ' \ar{h} \cY)$, the following equality of 2-cells holds:
\begin{eqD}{cond}
	\scalebox{0.9}{
	\begin{cdN}
		{} \& {\cP}  \arrow[dd,"{\theta_f}",Rightarrow ,shorten <=1.5ex, shorten >= 1.5ex] \arrow[rd,"{w}", bend left=20] \& {} \\[-6ex]
		{\cP \xp{\cY}{f}\cZ} \arrow[dd,"{F((g,\alpha))}"']  \arrow[ru,"{\pi\st_{\cP} f}", bend left=20] \arrow[rd,"{w_f}"', bend right=20] \&  \& {\cQ} \\[-6ex]
		{} \& {\cQ \xp{\cY}{f}\cZ} \arrow[d,"{\Xi_{(g,\alpha)}}",Rightarrow ,shorten <=1.5ex, shorten >= 1.5ex]  \arrow[ru,"{\pi\st_{\cQ} f}"'{inner sep=0.1ex}, bend right=20]\& {}\\
		{\cP \xp{\cY}{h}\cZ'} \arrow[r,"{w_h}"']\& {\cQ \xp{\cY}{h}\cZ'} \arrow[ruu,"{\pi\st_{\cQ} h}"', bend right=30]\& {} \\
	\end{cdN}}
	\h[3]=\h[3]
		\scalebox{0.9}{
	\begin{cdN}
		{} \& {\cP} \arrow[ddd,"{\theta_h}",Rightarrow ,shorten <=5.5ex, shorten >= 5.5ex]  \arrow[rd,"{w}", bend left=20] \& {} \\[-6ex]
		{\cP \xp{\cY}{f}\cZ} \arrow[dd,"{F((g,\alpha))}"']  \arrow[ru,"{\pi\st_{\cP} f}", bend left=20] \&  \& {\cQ} \\[-6ex]
		{} \& {}  \& {}\\
		{\cP \xp{\cY}{h}\cZ'} \arrow[ruuu,"{\pi\st_{\cP} h}"', bend left=15] \arrow[r,"{w_h}"']\& {\cQ \xp{\cY}{h}\cZ'} \arrow[ruu,"{\pi\st_{\cQ} h}"', bend right=30]\& {} \\
	\end{cdN}}
\end{eqD}
We now need to prove that $w\: \cP \to \cQ$ is the desired morphism. It is straightforward to prove that $w$ is a morphism in $\qst{\cX}{\cG}(\cY)$. Moreover, given $\cZ \ar{f} \cY\in S$, we can induce the desired 2-cell $\psi^{f}\: w_f \Rightarrow f\st w$ applying the two-dimensional universal property of the iso-comma object $\cQ \xp{\cY}{f}\cZ$ to the compatible 2-cells
\begin{cd}
	{\cP \xp{\cY}{f} \cZ} \arrow[rd,"{\pi\st_{\cP}}"'] \arrow[rr,"{w_f}"] \arrow[dd,"{f\st w}"']\& {} \& {\cQ\xp{\cY}{f} \cZ} \arrow[ld,"{{{\theta}_f}^{-1}}", Rightarrow ,shorten <=1.5ex, shorten >= 1.5ex] \arrow[dd,"{\pi\st_{\cQ} f}"] \\[-4ex]
	{} \& {\cP} \arrow[rd,"{w}"'] \& {} \\[-4ex]
	{\cQ\xp{\cY}{f} \cZ} \arrow[rr,"{\pi\st_{\cQ} f}"'] \& {} \& {\cQ} 
\end{cd}
and 
\begin{cd}
	{\cP \xp{\cY}{f} \cZ}  \arrow[rrdd,"{f\st \pi_{\cP}}"',""{name=E}]\arrow[rr,"{w_f}"] \arrow[dd,"{f\st w}"']\& {} \& {\cQ\xp{\cY}{f} \cZ} \arrow[dd,"{f \st\pi_{\cQ}}"]  \arrow[to=E,"{\gamma_{w_f}}",Rightarrow ,shorten <=1.5ex, shorten >= 1.5ex]\\[-4ex]
	{} \& {} \& {} \\[-4ex]
	{\cQ\xp{\cY}{f} \cZ} \arrow[rr,"{f \st\pi_{\cQ}}"'] \& {} \& {\cZ} 
\end{cd}
Finally, one can prove that the required conditions are satisfied tanks to equalities of type (\ref{cond}) and the compatibility conditions satisfied by the descent datum. This concludes the proof that (M) holds.

It remains to prove that the condition (O) holds. This will be the trickiest condition to prove.

Consider a weak descent datum for S of elements of $\qst{\cX}{\cG}$ that assigns to every $\cZ \ar{f} \cY\in S$ a principal $\cG$-2-bundle 
$$\pi_f\: \cW_f \to \cZ\in \qst{\cX}{\cG}(\cZ)$$
 and to every 2-cell $\Gamma\: f \Rightarrow f'$ with $f,f'\: \cZ \to \cY$ a morphism 
 $$w_{\Gamma}\: \cW_f \Rightarrow \cW_{f'}$$
 in $\qst{\cX}{\cG}(\cZ)$. Moreover, for every chain $\cZ' \ar{g} \cZ \ar{f} \cY$ with $f\in S$, we have an equivalence 
 $$\phi^{f,g}\: \cW_{f\c g} \aequi g\st \cW_f$$
 in $\qst{\cX}{\cG}(\cZ)$ and all the other data required by the definition of weak descent datum (see \cite{2stacks}) satisfying the required axioms.
 
 We need to prove that there exist a principal $\cG$-2-bundle 
 $$\pi\: \cW \to \cY\in \qst{\cX}{\cG}(\cY)$$
 and for every $\cZ \ar{f} \cY\in S$ an equivalence 
 $$\psi^f\: W_f \to f\st W$$
 in $\qst{\cX}{\cG}(\cY)$ and all the other data required by the definition of weak descent datum satisfying the required conditions.
 
 We use the assignment of the weak descent datum to construct a pseudofunctor $$\Lambda \: \Groth{R} \to \psslice{\K}{\cY}$$
 Given an object $(\cZ,\cZ \ar{f}\cY)\in \Groth{R}$, we define $\Lambda((\cZ,\cZ \ar{f}\cY))$ as the morphism 
 $$\cW_f \ar{\pi_f} \cZ \ar{f} \cY$$
 Given a morphism $(g,\id{})\: (\cZ', \wt{f\c g}) \to (\cZ, f)$ in $\Groth{R}$, we define $\Lambda((g,\id{}))$  as the 2-cell
  \begin{cd}
  	{\cW_{\wt{f\c g}}} \arrow[rr,"{\phi^{f,g}}"] \arrow[rd,"{\pi_{\wt{f\c g}}}"',""{name=A}] \&[-3ex] {} \&[-3ex] {g\st \cW_f} \arrow[to=A,"{\gamma_{\phi^{f,g}}}",Rightarrow ,shorten <=2.5ex, shorten >= 2.5ex, shift left=-1ex] \arrow[ld,"{g\st \pi_f}"] \arrow[rr,"{\pi\st_f g}"]\&[-3ex] {} \&[-3ex] {\cW_f} \arrow[ld,"{\pi_f}"]  \arrow[llld,"{\lambda^{\pi_f,g}}",Rightarrow ,shorten <=9.5ex, shorten >= 9.5ex]\\
  	{} \& {\cZ'} \arrow[rr,"{g}"] \arrow[rd,"{\wt{f\c g}}"' ,""{name=F}] \& {} \& {\cZ} \arrow[to=F,"{\sigma^{-1}_{f,g}}"{inner sep=0.01ex}, Rightarrow ,shorten <=3.5ex, shorten >= 3.5ex] \arrow[ld,"{f}"]\& {}\\ 
  	{} \& {} \& {\cY} \& {} \& {}
  \end{cd}
Given a morphism $(\id{}, \Gamma)\: (\cZ, \cZ \ar{f} \cY) \to (\cZ, \cZ \ar{f'} \cY)$ in $\Groth{R}$, we define $\Lambda((\id{},\Gamma))$ as the 2-cell 
\begin{cd}
	{\cW_f} \arrow[rrrr,"{\cW_{\Gamma}}"]  \arrow[rd,"{\pi_{f}}"']\&[-3ex] {} \&[-3ex] {} \&[-3ex] {} \&[-3ex] {\cW_{f'}} \arrow[ld,"{\pi_{f'}}"] \arrow[lllld,"{\gamma_{\cW_{\Gamma}}}",Rightarrow ,shorten <=11.5ex, shorten >= 11.5ex]\\
	{} \& {\cZ} \arrow[rd,"{f}"',""{name=G}] \arrow[rr,"{}",equal] \&{} \&  {\cZ} \arrow[ld,"{f'}",""{name=H}] \arrow[ll,"{\Gamma^{-1}}",Rightarrow ,shorten <=3.5ex, shorten >= 3.5ex, shift left=4ex]\& {}\\
	{} \& {} \& {\cY} \& {} \& {}
\end{cd}
One can then construct the image of a generic morphism of $\Groth{R}$ using the factorization in terms of morphisms of type $(g,\id{})$ and $(\id{},\Gamma)$.
Finally, given a 2-cell $\delta\: (g,\alpha) \to (h,\beta)$ in $\Groth{R}$ with $(g,\alpha),(h,\beta)\: (\cZ, \cZ \ar{f} \cY) \to (\cZ',\cZ' \ar{f'} \cY)$, we define $\Lambda(\delta)$ as the 2-cell
\begin{cd}
	{} \& {\cW_{\wt{f'\c g}}} \arrow[dd,"{\cW_{R(\delta)_{f'}}}"']  \arrow[r,"{\phi^{f',g}}"] \& {g\st \cW_{f'}}  \arrow[rd,"{\pi\st_{f'} g}"] \arrow[dd,"{\qst{\cX}{\cG}(\delta)_{\cW_{f'}}}"] \arrow[ldd,"{(\alpha_{\delta})_{f'}}", Rightarrow ,shorten <=4.5ex, shorten >= 4.5ex]\& {} \\[-3ex]
	{\cW_f}  \arrow[ru,"{\cW_{\alpha}}"] \arrow[rd,"{\cW_{\beta}}"']\& {} \& {} \& {\cW_{f'}} \\[-3ex]
	{} \& {\cW_{\wt{f'\c h}}} \arrow[r,"{\phi^{f',g}}"'] \& {h\st \cW_{f'}} \arrow[ru,"{\pi\st_{f'} h}"'] \& {} 
\end{cd}
where $(\alpha_{\delta})_{f'}$ is the 2-cell assigned by the weak descent datum (see Definition 3.14 of \cite{2stacks}).

We now prove that $\Lambda$ is a pseudofunctor.

\noindent Given a chain of morphisms
$$(\cZ,\cZ \ar{f} \cY) \ar{(g,\alpha)} (\cZ',\cZ' \ar{f'} \cY) \ar{(g',\alpha')} (\cZ'',\cZ'' \ar{f''} \cY)$$
in $\Groth{R}$, the isomorphic 2-cell 
$$\varepsilon_{(g,\alpha), (g', \alpha')}\: \Lambda((g',\alpha'))\c \Lambda((g,\alpha)) \Rightarrow \Lambda((g',\alpha')\c (g,\alpha))$$ is given by the following pasting diagram:
\begin{cd}
	{} \&[-4ex] {\cW_{\wt{f'\c g}}} \arrow[rd,"{\cW_{\alpha' \star g}}"{inner sep=0.1ex},""'{name=D}] \arrow[d,"{}",aiso] \arrow[r,"{\phi^{f',g}}"] \&[-4ex] {g\st \cW_{f'}} \arrow[d,"{}",twoiso ,shorten <=1.5ex, shorten >= 1.5ex] \arrow[rd,"{g\st \cW_{\alpha}}"] \arrow[r,"{\pi\st_{f'}g}"] \&[-4ex] {\cW_{f'}} \arrow[r,"{\cW_{\alpha}}"] \&[-4ex] {\cW_{\wt{f'' \c g'}}} \arrow[r,"{\phi^{f'',g}}"] \&[-4ex] {(g')\st \cW_{f''}} \arrow[rd,"{\pi\st_{f''}g}"]\&[-4ex] {} \\
	{\cW_{f}} \arrow[ru,"{\cW_{\alpha}}"] \arrow[rd,"{\cW_{\zeta}}"'] \& {\cW_{f'\c g}} \arrow[from=D,"{}",twoiso] \arrow[r,"{}",aiso]\& {\cW_{\wt{\wt{(f'' \c g')}\c g}}} \arrow[ld,"{}",aiso] \arrow[r,"{\phi^{f''\c g',g}}"]\& {g\st \cW_{\wt{f''\c g'}}} \arrow[r,"{g\st \phi^{f'',g'}}"] \arrow[ru,"{}"] \& {g\st (g')\st \cW_{f''}} \arrow[ld,"{}",aeq]  \arrow[ru,"{}"] \arrow[llld,"{\beta^{g',g}_{f''}}",Rightarrow ,shorten <=11.5ex, shorten >= 11.5ex]\& {} \& {\cW_{f''}} \\
	{} \& {\cW_{\wt{f'' \c (g' \c g)}}} \arrow[rr,"{\phi^{f'',g'\c g}}"'] \& {} \& {(g' \c g)\st \cW_{f''}} \arrow[rrru,"{\pi\st_{f''}(g'\c g)}"', bend right=20] \& {} \& {} \& {} 
\end{cd}
where $\beta^{g',g}_{f''}$ is the weak cocycle condition on $f'',g',g$ and $\zeta$ is the 2-cell such that $(g,\alpha) \c (g',\alpha) = (g\c g', \zeta)$.
 Moreover, it is straightforward to prove that $\Lambda((\id{},\id{}))$ is isomorphic to the identity and that the required conditions on 2-cells are satisfied.
 
 Let then $W$ be the sigma-bicolimit of the pseudofunctor $\opn{dom} \c \Lambda$ (that exists because $\K$ has all sigma-bicolimits by hypothesis). Since $\Lambda$ is pseudofunctorial it corresponds to a sigma-bicocone over $\cY$ of shape $\Groth{R}$ and so, using the universal property of the sigma-bicolimit, we can induce a morphism 
 $$\pi\: \cW \to \cY$$
 Moreover, by the calculus of colimits in lax slices developed by Mesiti in \cite{Mesiticolimits} (and more in detail in \cite{Mesitithesis} in the sigma-bicolimit case), we obtain the following equality
 $$\pi= \sigmabicolim{\Lambda}$$
 
 We now prove that the morphism $\pi$ is a principal $\cG$-2-bundle. $\pi$ is trivially $\cG$-equivariant (since we are considering trivial actions of $\cG$) so it remains to prove that it is 2-locally trivial. 
 Given $\cZ \ar{f} \cY\in S$, since $\pi_f\: \cW_f\to \cZ$ is 2-locally trivial, there exists a covering bisieve $S_f\in \tau(\cZ)$ such that for every $\cT \ar{h} \cZ\in S_f$ the iso-comma object of $\pi_f$ and $h$ is equivalent to the product $\cG \x \cT$. Consider then the collection 
 $$S'=\{ \cT \ar{h} \cZ \ar{f} \cY | f\in S, h\in S_f\}$$
 Using the axiom $(T2)$ of bitopology, it is easy to prove that $S'$ is a covering bisieve over $\cY$.  Given a morphism $\cT\ar{l} \cY \in S'$ that is equal to the composite $\cT \ar{r} \cZ' \ar{g} \cY$, we obtain the following chain of equalities
 $$l\st \pi= l\st \sigmabicolim{\Lambda}= \sigmabicolim{l\st \Lambda},$$
 where the second equality is given by the hypothesis that iso-comma objects preserve colimits.
 And given an object $(\cZ,f)\in \Groth{R}$ the following is an iso-comma object square
 \begin{cd}
 	{\cG \x \cT} \arrow[r,"{}",iso,shift left=-5ex]\arrow[r,"{q_1}"] \arrow[d,"{q_2}"']\& {\cW_f \xp{\cZ}{}(\cZ \x[\cY] \cZ')} \arrow[r,"{}",iso,shift left=-5ex] \arrow[r,"{}"] \arrow[d,"{(f\st g)\st \pi_f}"]\& {\cW_f} \arrow[d,"{\pi_f}"]\\
 	{(\cZ \x[\cY] \cZ')\x[\cZ']\cT} \arrow[r,"{}",iso,shift left=-5ex] \arrow[r,"{(f\st g)\st r}"] \arrow[d,"{}"]\& {\cZ \x[\cY] \cZ'} \arrow[r,"{}",iso,shift left=-5ex] \arrow[d,"{g\st f}"]  \arrow[r,"{f\st g}"]\& {\cZ} \arrow[d,"{f}"]\\
 	{\cT} \arrow[r,"{r}"'] \& {\cZ'} \arrow[r,"{g}"'] \& {\cY}
 \end{cd}
thanks to the fact that the covering bisieve generated by $(f\st g)\st h$ with $h\in S_g$ trivializes $(f\st g)\st \pi_f$ as shown in the proof of Proposition \ref{comma2bun}. This implies that the iso-comma of the morphism $\cW_f \ar{\pi_f} \cZ \ar{f} \cY$ along $l$ is equivalent to the product $\cG \x \cT$. And so the morphism $\pi$ is 2-locally trivial. 

Furthermore, we can induce a morphism $\alpha\: \cW \to \cX$ applying the universal property of the sigma-bicolimit to the sigma-bicocone of shape $\Groth{R}$ whose structure 2-cells are simply given by the fact that the morphisms involved are morphisms in 2-categories of the form $\qst{\cX}{\cG}(\cT)$ with $\cT\in\K$. This conclude the proof of the fact that $(\cW \ar{\pi} \cY, \cW \ar{\alpha} \cX)$ is in $\qst{\cX}{\cG}(\cY)$.  

We now need to induce the equivalence $\psi^{f}\: \cW_{f} \aequi f\st \cW$ for every $f\in S$. We equivalently induce the pseudoinverse $\theta^{f}\: f\st \cW \to \cW_f$.
We observe that the following chain of equalities holds
$$f\st \cW = f\st \sigmabicolim{\Lambda}= \sigmabicolim{f\st \Lambda}$$
and so we can use the universal property of the sigma-bicolimit to induce $\theta^{f}$. To do so, we construct a sigma bicocone $\Sigma$ over $\cW_f$. Given $(\cZ', \cZ' \ar{g} \cY)\in \Groth{R}$, we define $\Sigma_g$ as the composite
$$f\st \cW_g \ar{\lambda_{f,g}} (g\st f)\st \cW_{g} \ar{\rho^{g,g\st f}} \cW_{g\c g\st f} \ar{\cW_{\sigma_{f,g}}} \cW_{f\c f\st g} \ar{\phi^{f,f\st g}} (f\st g)\st \cW_{f} \ar{\pi\st_{f} (f\st g)} \cW_f,$$
where $\lambda_{f,g}$ is the isomorphism induced by the universal property of $(g\st f)\st \cW_g$ using the 2-cell of iso-comma object of $f$ and $g\c \pi_g$ and $\rho^{g,g\st f}$ is the pseudoinverse of $\psi^{g,g\st f}$.

Given a morphism $(\id{},\alpha)\: (\cZ',g) \to (\cZ', g')$ in $\Groth{R}$, the structure 2-cell $\Sigma_{(\id{},\alpha)}$ is given by the pasting diagram
\begin{eqD*}
	\scalebox{0.8}{
		\begin{cdN}
			{f\st \cW_g} \arrow[dd,"{f\st \cW_{\alpha}}"'] \arrow[r,"{\lambda_{f,g}}"] \& {(g\st f)\st \cW_{g}} \arrow[dd,"{\delta_{\alpha}}"] \arrow[r,"{\rho^{g,g\st f}}"] \& {\cW_{g\c g\st f}} \arrow[r,"{\cW_{\sigma_{f,g}}}"]  \& {\cW_{f\c f\st g}} \arrow[r,"{\phi^{f,f\st g}}"]  \& {(f\st g)\st \cW_{f}} \arrow[rd,"{\pi\st_{f} (f\st g)}"] \arrow[dd,"{\pi\st_{f}(f\st \alpha)}"] \arrow[llldd,"{\zeta}",twoiso ,shorten <=20.5ex, shorten >= 20.5ex]\& {} \\
			{} \& {} \& {} \& {} \& {} \& {\cW_f} \\
			{f\st \cW_{g'}} \arrow[r,"{\lambda_{f,g'}}"'] \& {((g')\st f)\st \cW_{g'}} \arrow[r,"{\rho^{g',(g')\st f}}"'] \& {\cW_{g'\c (g')\st f}} \arrow[r,"{\cW_{\sigma_{f,g'}}}"']  \& {\cW_{f\c f\st g'}} \arrow[r,"{\phi^{f,f\st g'}}"']  \& {(f\st g')\st \cW_{f}} \arrow[ru,"{\pi\st_{f} (f\st g')}"']\& {} \\
	\end{cdN}}
\end{eqD*}
where the morphism $\delta_{\alpha}$ is induced by the universal property of the iso-comma object $((g')\st f)\st W_{g'}$ using the 2-cell 
\begin{cd}
	{(g\st f)\st \cW_{g}} \arrow[r,"{}"] \arrow[d,"{}"]\& {\cW_{g}} \arrow[ld,"{\lambda^{\pi_g,g\st f}}"',Rightarrow,shorten <=4.5ex, shorten >= 4.5ex, shift left=-1ex] \arrow[r,"{\cW_{\alpha}}"] \arrow[rd,"{\pi_g}"'{pos=0.4},""{name=D}]\& {\cW_{g'}} \arrow[d,"\pi_{g'}"]  \arrow[from=D,"{}",iso] \\
	{\cZ \xp{\cY}{g}\cZ'} \arrow[r,"{f\st \alpha}"'] \arrow[rr,"{g\st f}",bend left=24,""{name=H}] \& {\cZ \xp{\cY}{g'}\cZ'} \arrow[r,"{}"]  \& {\cZ'} 
\end{cd}
and the 2-cell $\zeta$ is given by the following pasting diagram
\begin{eqD*}
	\scalebox{0.75}{
		\begin{cdN}
			{(g\st f)\st \cW_g} \arrow[rd,"{}"] \arrow[dddd,"{\delta_{\alpha}}"'] \arrow[rr,"{\rho^{g,g\st f}}"] \&[-7ex] {} \&[-2ex] {\cW_{\wt{g\c g\st f}}}  \arrow[d,"{}"] \arrow[r,"{\cW_{\sigma_{f,g}}}"] \&[-0.5ex] {\cW_{f\c f\st g}} \arrow[d,"{}",aiso] \arrow[rr,"{\phi^{f,f\st g}}"] \&[-4ex] {} \&[-4ex] {(f\st g)\st \cW_f} \arrow[dddd,"{\pi\st _f(f\st \alpha)}"]\\[4ex]
			{} \& {((g')\st f)\st \cW_{g}} \arrow[ru,"{}",iso] \arrow[r,"{\rho^{g',(g')\st f \c f\st \alpha}}"] \arrow[dd,"{}",aiso]\& {\cW_{\wt{g'\c (g\st f \c f\st \alpha)}}} \arrow[ru,"{}",iso]\arrow[r,"{}"] \arrow[d,"{}",aiso] \& {\cW_{\wt{\wt{f\c f\st g'} A\c f\st \alpha}}} \arrow[dd,"{\phi^{f\c f\st g,f\st \alpha}}"] \arrow[rru,"{\beta^{f\st g',f\st \alpha}_f}",iso]\& {(f\st \alpha)\st (f\st g')\st \cW_f)} \arrow[ru,"{}",aeq] \arrow[rddd,"{}"] \& {} \\[-2ex]
			{} \& {} \& {\cW_{\wt{\wt{(g'\c (g')\st f)}\c f\st \alpha}}} \arrow[d,"{\phi^{g'\c (g')\st f,f\st \alpha}}",bend left=20,""{name=A}] \& {} \& {} \& {}\\[3ex]
			{} \& {(f\st \alpha)\st ((g')\st f)\cW_g} \arrow[ruu,"{\beta^{(g')\st f,f\st \alpha}_{g'}}"{inner sep=0.3ex},iso] \arrow[r,"{(f\st \alpha)\st \rho^{g',(g')\st f}}"'{inner sep=1ex}] \arrow[ld,"{}"]\& {(f\st \alpha)\st \cW_{\wt{g \c (g')\st f}}} \arrow[r,"{(f\st \alpha)\st \cW_{\sigma_{f,g'}}}"'{inner sep=1ex}] \arrow[d,"{}"] \arrow[u,"{\rho^{g'\c (g')\st f,f\st \alpha}}",bend left=20,""{name=B}] \arrow[from=B,to=A,"{}",iso] \arrow[ruu,"{}",iso] \& {(f\st\alpha)\st \cW_{\wt{f\c (f')\st g}}} \arrow[d,"{}"] \arrow[ruu,"{(f\st \alpha)\st \phi^{f,f\st g'}}"'] \& {} \& {} \\[4ex]
			{((g')\st f)\st \cW_{g'}} \arrow[ruuuu,"{}",iso] \arrow[rr,"{\rho^{g',{g'}\st f}}"] \&[-1ex] {} \&[-1ex] {\cW_{\wt{(g')\c (g')\st f}}} \arrow[r,"{\cW_{\sigma_{f,g'}}}"] \&[-0.5ex] {\cW_{f\c f\st g'}} \arrow[rr,"{\phi^{f,f\st g'}}"] \&[-3ex] {} \&[-3ex] {(f\st g')\st \cW_f}
	\end{cdN}}
\end{eqD*}

Given a morphism $(h, \id{})\: (\cZ'',\wt{g\c h}) \to (\cZ', g)$ in $\Groth{R}$, the structure 2-cell $\Sigma_{(h,\id{})}$ is given by the pasting diagram
\begin{eqD*}
	\scalebox{0.8}{
		\begin{cdN}
			{f\st \cW_{\wt{g\c h}}} \arrow[d,"{f\st \phi^{g,h}}"'] \arrow[r,"{\lambda_{f,{\wt{g\c h}}}}"] \&[-2ex] {(({\wt{g\c h}})\st f)\st \cW_{{\wt{g\c h}}}} \arrow[dd,"{\iota_{g,h}}"] \arrow[r,"{\rho^{{\wt{g\c h}},({\wt{g\c h}})\st f}}"] \&[-2ex] {\cW_{({\wt{g\c h}})\c(( {\wt{g\c h}})\st f)}} \arrow[r,"{\cW_{\sigma_{f,{\wt{g\c h}}}}}"]  \&[-2ex] {\cW_{f\c f\st( {\wt{g\c h}})}} \arrow[r,"{\phi^{f,f\st ({\wt{g\c h}})}}"]  \&[-2ex] {(f\st ({\wt{g\c h}}))\st \cW_{f}} \arrow[llldd,"{\xi}",twoiso ,shorten <=22.5ex, shorten >= 22.5ex] \arrow[rd,"{\pi\st_{f}(f\st(\wt{g\c h}))}"] \arrow[dd,"{l_{g,h}}"]\&[-2ex] {} \\
			{f\st h\st \cW_g} \arrow[d,"{f\st \pi\st _g h}"'] \& {} \& {} \& {} \& {} \& {\cW_f}\\
			{f\st \cW_g} \arrow[r,"{\lambda_{f,g}}"'] \& {(g\st f)\st \cW_{g}} \arrow[r,"{\rho^{g,g\st f}}"'] \& {\cW_{g\c g\st f}} \arrow[r,"{\cW_{\sigma_{f,g}}}"']  \& {\cW_{f\c f\st g}} \arrow[r,"{\phi^{f,f\st g}}"']  \& {(f\st g)\st \cW_{f}} \arrow[ru,"{\pi\st_{f} (f\st g)}"']\& {} 
	\end{cdN}}
\end{eqD*}
where the morphism $\iota_{g,h}$ is induced by the universal property of the iso-comma object $(g\st f)\st W_g$ using the 2-cell 
\begin{cd}
	{((\wt{g\c h})\st f)\st \cW_{\wt{g\c h}}} \arrow[r,"{}"]  \arrow[dd,"{}"]\& {\cW_{\wt{g\c h}}} \arrow[ldd,"{\lambda^{g\c \pi_g,f}}",Rightarrow ,shorten <=9.5ex, shorten >= 9.5ex]\arrow[rd,"{}",iso,xshift=-5ex, shift left=2ex] \arrow[r,"{\phi^{g,h}}"] \arrow[d,"{\pi_{\wt{g\c h}}}"'] \& {h\st \cW_g} \arrow[ld,"{h\st \pi_g}"{inner sep=0.1ex},""{name=N}] \arrow[d,"{\pi_g}"]\\
	{} \& {\cZ''} \arrow[r,"{h}"] \arrow[rd,"{\wt{g\c h}}"' ,""{name=G}] \& {\cZ'} \arrow[to=G,"{\sigma^{-1}_{g,h}}",Rightarrow] \arrow[d,"{g}"] \\
	{\cZ} \arrow[rr,"{f}"'] \& {} \& {\cY} 
\end{cd}
the morphism $l_{g,h}$ is induced by the universal property of the iso-comma object $(f\st g)\st W_f$ using the 2-cell
\begin{cd}
	{(f\st (\wt{g\c h}))\st W_f}  \arrow[r,"{}"] \arrow[d,"{}"]\& {\cW_f} \arrow[dd,"{\pi_f}"]  \arrow[ld,"{\lambda^{\pi_f,f\st (\wt{g\c h})}}"{inner sep=0.1ex},Rightarrow,shorten <=5.5ex, shorten >= 5.5ex]\\
	{f\st \cZ''}  \arrow[rd,"{f\st (\wt{g\c h})}",""{name=D}] \arrow[d,"{f\st h}"']\& {}\\
	{f\st \cZ'} \arrow[from=D,"{f\st \sigma_{g,h}}",Rightarrow ,shorten <=1.5ex, shorten >= 1.5ex]\arrow[r,"{f\st g}"'] \& {\cZ} 
\end{cd}
the commutativity of the diagram on the left can be checked on components and the 2-cell $\xi$ is obtained as a pasting of appropriate weak cocycle condition 2-cells and cells given by the assignment of the descent datum in a way similar to the one used to construct $\zeta$ in $\Sigma_{(\id{},\alpha)}$.
One can then prove that these data give a sigma-bicocone $\Sigma$. And we can thus induce the desired morphism $\theta^f$. Furthermore, since the morphism 
$$f\st \cW_g \ar{\lambda_{f,g}} (g\st f)\st \cW_{g} \ar{\rho^{g,g\st f}} \cW_{g\c g\st f} \ar{\cW_{\sigma_{f,g}}} \cW_{f\c f\st g} \ar{\phi^{f,f\st g}} (f\st g)\st \cW_{f}$$
is an equivalence, it follows that $\theta^f$ is an equivalence as well. And so we have induced the desired equivalence $\psi^f\: \cW_f \aequi f\st \cW$ (pseudoinverse of $\theta^{f}$).

We now need to prove that, given $\cZ \ar{f} \cY\in S$ and a morphism $\cZ' \ar{g} \cZ$, there exists an invertible 2-cell 
	\begin{cd}
	{\cW_{\wt{f\c g}}} \arrow[rr,"{\phi^{f,g}}"] \arrow[d,"{\psi^{\wt{f\c g}}}"']\& {} \& {g\st \cW_f} \arrow[d,"{g\st \psi^f}"] \arrow[lld,"{(\varepsilon_g)_f}",Rightarrow ,shorten <=10.5ex, shorten >= 12.5ex]\\
	{(\wt{f\c g})\st \cW} \arrow[r,"{}",aiso]\& {(f\c g)\st \cW}\arrow[r,"{}",aeq] \& {g\st f\st \cW} 
\end{cd}
We equivalently prove that there exists an invertible 2-cell
\begin{cd}
	{g\st f\st \cW} \arrow[r,"{g\st \theta^f}"] \arrow[d,"{}",aeq]\& {g\st \cW_f} \arrow[r,"{\rho^{f,g}}"]  \& {\cW_{\wt{f\c g}}} \arrow[dd,"{}",equal] \arrow[lldd,"{(\iota_g)_f}",Rightarrow ,shorten <=8.5ex, shorten >= 8.5ex] \\[-4ex]
	{(f\c g)\st \cW} \arrow[d,"{}",aiso] \& {} \& {}\\[-4ex]
	{(\wt{f\c g})\st \cW} \arrow[rr,"{\theta^{\wt{f\c g}}}"'] \& {} \& {\cW_{\wt{f\c g}}} 
\end{cd}
We notice that 
$$g\st f\st \cW =g\st f\st \sigmabicolim{(\opn{dom}\c\Lambda)}=\sigmabicolim{(g\st \c f\st \c\opn{dom}\c\Lambda)}$$
and so we can induce the 2-cell $(\iota_g)_f$ via the universal property of the sigma-bicolimit. Given $(\cT, \cT \ar{t} \cY)\in \Groth{R}$, we define a 2-cell $\Gamma_t$ as the following pasting diagram
\begin{eqD*}
	\scalebox{0.75}{
		\begin{cdN}
			{g\st f\st \cW_t} \arrow[rdddd,"{}",aeq]\arrow[rrr,"{g\st f\st \sigma_t}"] \arrow[ddd,"{g\st f\st \sigma_t}"] \arrow[rrd,"{g\st \lambda_{f,t}}"] \&[-3ex] {}\&[-3ex] {} \&[-3ex] {g\st f\st \cW} \arrow[rrr,"{g\st \theta^f}"]\&[-2ex] {} \&[-2ex] {} \&[-2ex] {g\st \cW_f} \arrow[ddddddddd,"{\rho^{f,g}}"] \\
			{} \& {} \& {g\st (t\st f)\st \cW_t} \arrow[r,"{g\st \rho^{t,t\st f}}"] \arrow[d,"{}",aiso] \& {g\st \cW_{\wt{t\c t\st f}}} \arrow[r,"{}",iso, shift left =6ex] \arrow[r,"{g\st \cW_{\sigma_{f,t}}}"] \arrow[d,"{}",aiso]\& {g\st \cW_{\wt{f\c f\st t}}} \arrow[r,"{g\st \phi^{f,f\st t}}"] \arrow[d,"{}",aiso] \& {g\st (f\st t)\st \cW_f} \arrow[ru,"{g\st (\pi\st _f (f\st t))}"] \arrow[d,"{}",aiso] \& {} \\
			{} \& {} \& {m\st (t\st f)\st \cW_t} \arrow[r,"{\beta^{t\st f,m}_t}",iso, shift left=-12ex] \arrow[r,"{m\st \rho ^{t,t\st f}}"] \arrow[dd,"{}",aiso] \& {m\st \cW_{\wt{t\c t\st f}}} \arrow[r,"{}"], \arrow[r,"{}",iso, shift left=-4ex] \arrow[d,"{\rho^{\wt{t\c t\st f},m}}"]\& {m\st \cW_{f\c f\st t}} \arrow[r,"{\beta^{f\st t,m}_f}",iso, shift left=-12ex] \arrow[r,"{m\st \phi^{f,f\st t}}"] \arrow[d,"{\rho^{\wt{f\c f\st t},m}}"'{}, bend right=15,""{name=A}] \& {m\st (f\st t)\st \cW_f} \arrow[dd,"{}",aiso] \& {} \\
			{g\st f\st \cW} \arrow[d,"{}",aiso] \& {} \& {} \& {\cW_{\wt{\wt{t\c t\st f}\c m}}} \arrow[r,"{}",iso, shift left=-6ex] \arrow[r,"{}"] \arrow[d,"{}",aiso] \& {\cW_{\wt{{f\c f\st t}\c m}}} \arrow[d,"{}",aiso] \arrow[u,"{\phi^{\wt{f\c f\st t},m}}"'{}, bend right=15,""{name=B}] \arrow[from= B,to =A,"{}",iso]\& {} \& {} \\
			{(f\c g) \st \cW} \arrow[ddddd,"{}",aiso]  \&{(f\c g)\st \cW_t}  \arrow[l,"{(f\c g) \st \sigma_t}"] \arrow[dddd,"{}",aiso] \arrow[r,"{\lambda_{f\c g,t}}"] \& {(t\st (f\c g))\st \cW_t} \arrow[r,"{}",iso, shift left=-22ex] \arrow[dddd,"{}"] \arrow[r,"{\rho^{t,t\st f \c m}}"]  \& {\cW_{\wt{t\c t\st (\wt{f\c g})}}} \arrow[r,"{}",iso, shift left=-22ex]\arrow[dddd,"{}"] \arrow[r,"{}"]\& {\cW_{\wt{f\c f\st t \c m}}} \arrow[r,"{}",iso, shift left=-6ex] \arrow[r,"{\phi^{f,f\st t \c m}}"] \arrow[d,"{}"] \& {(f\st t \c m)\st \cW_f} \arrow[d,"{}"]\& {} \\
			{} \& {} \& {} \& {} \& {\cW_{\wt{f\c g \c (f\c g) \st t}}} \arrow[r,"{\beta^{g, (f\c g) \st t}_{f}}",iso, shift left=-10ex] \arrow[r,"{\phi^{f,g\c (f\c g) \st t}}"] \arrow[dd,"{}",aiso] \& {(g\c (f\c g)\st t) \st \cW_f} \arrow[d,"{}",aiso] \& {} \\
			{} \& {} \& {} \& {} \& {} \& {((f\c g) \st t)\st (g\st \cW_f)}  \arrow[d,"{((f\c g) \st t)\st\rho^{f,g}}"{pos=0.4}, bend left=20 ,""{name=F}] \& {} \\
			{} \& {} \& {} \& {} \& {\cW_{\wt{\wt{f\c g}\c (f\c g) \st t}}} \arrow[r,"{}",iso, shift left=-6ex] \arrow[r,"{\phi^{\wt{f\c g},(f\c g) \st t}}"] \arrow[d,"{}"] \& {((f\c g) \st t)\st \cW_{\wt{f\c g}}} \arrow[d,"{}"] \arrow[u,"{((f\c g) \st t)\st\phi^{f,g}}"{pos=0.6}, bend left=20 ,""{name=G}] \arrow[from= G,to =F,"{}",iso]\& {} \\
			{} \& {(\wt{f\c g})\st \cW_t} \arrow[r,"{\lambda_{\wt{f\c g},t}}"] \arrow[ld,"{(\wt{f\c g})\st \sigma_t}"] \& {(t\st (\wt{f\c g}))\st \cW_t} \arrow[r,"{\rho^{t,t\st(\wt{f\c g})}}"]\& {\cW_{\wt{t\c t\st (\wt{f\c g})}}} \arrow[r,"{}",iso, shift left=-6ex] \arrow[r,"{\cW_{\wt{f\c g},t}}"]\& {\cW_{\wt{(\wt{f\c g})\c (\wt{f\c g})\st t}}} \arrow[r,"{\phi^{\wt{f\c g},(\wt{f\c g})\st t}}"] \& {((\wt{f\c g})\st t)\st \cW_{f\c g}} \arrow[rd,"{}"] \& {} \\
			{(\wt{f\c g})\st \cW} \arrow[rrrrrr,"{\theta^{\wt{f\c g}}}"'] \& {} \& {} \& {} \& {} \& {} \& {\cW_{\wt{f\c g}}} 
	\end{cdN}}
\end{eqD*}
where $\varepsilon$ is the canonical equivalence between $(f\c g)\st \cT$ and $g\st f\st \cT$ and $m=(f\st t)\st g \c \varepsilon$. Using the compatibility conditions of the weak descent datum, one can then prove that the $\Gamma_t$'s form a modification and so the 2-cell $(\iota_g)_f$ is induced by the universal property of the sigma-bicolimit.

We now need to prove that, given a 2-cell $\gamma\: f\Rightarrow f'$ with $f,f'\: \cZ\to \cY$ in $R(\cZ)$  there exists an isomorphic 2-cell
\begin{eqD*}
	\begin{cdN}
		{\cW_f} \arrow[r,"{\psi^f}"] \arrow[d,"{\eta_{\gamma}}"']\& {f\st \cW} \arrow[d,"{{\qst{\cX}{\cG}(\gamma) _\cW}}"] \arrow[ld,"{\psi_{\gamma}}",twoiso, shorten <= 2ex, shorten >= 2ex]\\
		{\cW_{f'}} \arrow[r,"{\psi^{f'}}"'] \& {{f'}\st \cW}
	\end{cdN}
\end{eqD*}

We equivalently induce an isomorphic 2-cell
\begin{cd}
	{f\st \cW} \arrow[r,"{\theta^f}"] \arrow[d,"{\qst{\cX}{\cG}(\gamma)_{\cW}}"']\& {\cW_f} \arrow[d,"{\eta_\gamma}"] \arrow[ld,"{\kappa_\gamma}",twoiso ,shorten <=2.5ex, shorten >= 2.5ex] \\
	{{f'}\st \cW}  \arrow[r,"{\theta^{f'}}"']\& {\cW_{f'}}
\end{cd}
We notice that 
$$f\st \cW=f\st \sigmabicolim{\opn{dom}\c \Lambda}= \sigmabicolim{f\st \c \opn{dom}\c \Lambda}$$
and so we can induce the desired 2-cell $\kappa_{\gamma}$ via the universal property of the sigma-bicolimit using a modification $\Delta$. Given $(\cT, \cT \ar{t} \cY)\in \Groth{R}$, we define the 2-cell 
\begin{cd}
	{f\st \cW_t}  \arrow[r,"{f\st \sigma_t}"] \arrow[d,"{f\st \sigma_t}"']\& {f\st \cW} \arrow[r,"{\theta^f}"] \& {\cW_f} \arrow[d,"{\eta_{\gamma}}"] \arrow[lld,"{\Delta_t}",twoiso ,shorten <=5.5ex, shorten >= 5.5ex]  \\
	{f\st \cW} \arrow[r,"{\qst{\cX}{\cG}(\gamma)_{\cW}}"'] \& {{f'}\st \cW} \arrow[r,"{\theta^{f'}}"']\& {\cW_{f'}} 
\end{cd}
as a pasting diagram analogous to the one used to define $\Gamma_t$. And it is easy to see that we obtain a modification $\Delta$ that induces the desired 2-cell $\kappa_{\gamma}$.

Finally, it is straightforward to verify that the defined data satisfy the axioms required by the definition of weakly effective weak descent datum. This concludes the proof of the condition (O).
\end{proof}

\bibliographystyle{abbrv} 
\bibliography{Bibliography_quotient2stacks.bib}
\end{document}